\documentclass[12pt,reqno]{amsart}
\usepackage{amsmath,amssymb,mathrsfs,amsthm}
\usepackage{amsfonts}
\usepackage[inline]{enumitem} 
\usepackage{braket}
\usepackage[utf8]{inputenc} 
\usepackage[sharp]{easylist}
\usepackage{hyperref}
\hypersetup{%
	colorlinks=true, linkcolor=blue,
	citecolor=green
}

\usepackage[paper=letterpaper,margin=1in]{geometry}

\newtheorem{theorem}{Theorem}[section]
\newtheorem{lemma}[theorem]{Lemma}
\newtheorem{proposition}[theorem]{Proposition}
\newtheorem{corollary}[theorem]{Corollary}
\theoremstyle{definition}
\newtheorem{definition}[theorem]{Definition}
\newtheorem{remark}[theorem]{Remark}

\numberwithin{equation}{section}
\allowdisplaybreaks

\usepackage{acronym}
\acrodef{SHE}[SHE]{Stochastic Heat Equation}
\acrodef{KPZ}[KPZ]{Kardar--Parisi--Zhang}
\acrodef{ASEP}[ASEP]{Asymmetric Simple Exclusion Process}
\acrodef{BDG}[BDG]{Burkholder--Davis--Gundy}
\acrodef{EP}[HSEP]{Higher Spin Exclusion Process}

\newcommand{\move}{ K }
\newcommand{\vecmove}{\vec{K}}
\newcommand{\moveCent}{ \overline{K} }
\newcommand{\moveI}{ I }
\newcommand{\moveICent}{ \overline{I} }
\newcommand{\moveIp}{ \widetilde{I} }

\newcommand{\B}{ B }		
\newcommand{\Bp}{ B' }		

\newcommand{\RWp}{ R' }			
\newcommand{\RW}{ R }			
\newcommand{\hkp}{ p' }			
\newcommand{\hk}{p}				
\newcommand{\hke}{p_\varepsilon}
\newcommand{\genps}{\widehat{\mathcal{L}}_\varepsilon}		
\newcommand{\gen}{\mathcal{L}_\varepsilon}		

\newcommand{\mg}{ M }		
\newcommand{\mgg}{ W }		
\newcommand{\mggp}{ W' }	

\newcommand{\norm}[2]{ |#1|_{#2} }		
\newcommand{\normd}[2]{ |#1|^*_{#2} }	

\newcommand{\Xsp}{ \mathbb{X} }
\newcommand{\Gsp}{ \mathbb{G} }		

\newcommand{\Tran}{ \mathcal{T} }	

\newcommand{\img}{\mathbf{i}}			

\newcommand{\bd}{ \mathcal{B}_\e }			
\newcommand{\er}{ \mathcal{E}_\e }			

\newcommand{\qv}{ \Theta }	
\newcommand{\qvp}{ \Theta' }	
\newcommand{\const}{ \gamma }
\newcommand{\drift}{ \mu }

\newcommand{\Drift}{ \widehat{\mu} }
\newcommand{\dens}{ \rho }
\newcommand{\cent}{ \lambda }

\newcommand{\Cent}{ \widehat{\lambda} }

\newcommand{\var}{ \sigma }
\newcommand{\xcent}{ r_* }

\newcommand{\qe}{ q_\varepsilon }
\newcommand{\drifte}{ \mu_\varepsilon }

\newcommand{\Drifte}{ \widehat{\mu}_\varepsilon }
\newcommand{\cente}{ \lambda_\varepsilon }
\newcommand{\Cente}{ \widehat{\lambda}_\varepsilon }

\newcommand{\tcente}{ \tau^{\varepsilon}_* }
\newcommand{\vare}{ \sigma_\varepsilon }

\newcommand{\e}{ \varepsilon }

\newcommand{\Ne}{N^\varepsilon}
\newcommand{\hatNe}{\widehat{N}^\varepsilon}

\newcommand{\Ze}{Z_\varepsilon}

\newcommand{\alphae}{\alpha_\varepsilon}
\newcommand{\alphaej}[1]{\alpha^\varepsilon_{#1}}

\newcommand{\aej}[1]{a^{\varepsilon}_{#1}}
\renewcommand{\ne}{n_\varepsilon}
\newcommand{\te}{t_\varepsilon}

\newcommand{\vecxe}{\vec{x}^{\varepsilon}}
\newcommand{\aux}{D}

\newcommand{\vecg}{\vec{g}}

\newcommand{\vecx}{\vec{x}}
\newcommand{\vecy}{\vec{y}}

\newcommand{\veczeta}{\vec{\zeta}}

\newcommand{\Hlim}{\mathcal{H}}								
\newcommand{\Hdlim}{\widetilde{\mathcal{H}}}				

\newcommand{\Hx}{H^J_\varepsilon}
\newcommand{\Hdx}{\widetilde{H}^J_\varepsilon}

\newcommand{\Zlim}{\mathcal{Z}}								
\newcommand{\Zic}{\mathcal{Z}^{\text{ic}}}								
\newcommand{\Zdr}{Z_\text{dr}}						
\newcommand{\Zmg}{Z_\text{mg}}

\newcommand{\Zdlim}{\widetilde{\mathcal{Z}}}				
\newcommand{\Zd}{\widetilde{Z}}						
\newcommand{\Zde}{\widetilde{Z}_\varepsilon}	
\newcommand{\Zddr}{\widetilde{Z}_\text{dr}}		
\newcommand{\Zdmg}{\widetilde{Z}_\text{mg}}

\newcommand{\hatN}{\widehat{N}}

\DeclareMathOperator{\rmdd}{mod}
\newcommand{\rmd}{\rmdd_J}

\DeclareMathOperator{\Ex}{\mathbf{E}}	
\DeclareMathOperator{\Cov}{Cov}			
\DeclareMathOperator{\Pro}{\mathbf{P}}	
\DeclareMathOperator{\ind}{\mathbf{1}}	

\DeclareMathOperator{\Ber}{Ber}			
\newcommand{\diff}{\frac{d~}{dx}}
\newcommand{\N}{ \mathbb{N} }
\newcommand{\bbC}{ \mathbb{C} }
\newcommand{\bbR}{ \mathbb{R} }
\newcommand{\bbZ}{ \mathbb{Z} }
\newcommand{\Pair}{ \mathscr{P} }

\newcommand{\scrF}{ \mathscr{F} }

\usepackage{graphicx}
\newcommand*{\Cdot}{{\raisebox{-0.5ex}{\scalebox{1.8}{$\cdot$}}}} 

\newcommand{\BK}[1]{ {\left( #1 \right)} }
\newcommand{\sqBK}[1]{ {\left[ #1 \right]} }
\newcommand{\curBK}[1]{ {\left\{ #1 \right\}} }
\newcommand{\absBK}[1]{ {\left| #1 \right|} }
\newcommand{\VertBK}[1]{ {\left\Vert #1 \right\Vert} }
\newcommand{\BVert}{ \Big\Vert }
\newcommand{\bVert}{ \big\Vert }
\newcommand{\Vertbk}[1]{ \Vert #1 \Vert }
\newcommand{\angleBK}[1]{ {\left < #1 \right >} }
\newcommand{\anglebk}[1]{ \langle #1 \rangle }

\newcommand{\floorbk}[1]{ \lfloor #1 \rfloor }

\begin{document}

\title[KPZ equation of higher-spin exclusion processes]{KPZ equation limit of higher-spin exclusion processes}
\author[I.\ Corwin]{Ivan Corwin}
\address{I.\ Corwin,
	Departments of Mathematics, Columbia University,
	\newline\hphantom{\quad \ I.\ Corwin}
	2990 Broadway, New York, NY 10027
	}
\email{ivan.corwin@gmail.com}

\author[L.-C.\ Tsai]{Li-Cheng Tsai}
\address{L.-C.\ Tsai,
	Departments of Mathematics, Columbia University,
	\newline\hphantom{\quad \ L.-C. Tsai}
	2990 Broadway, New York, NY 10027
	}
\email{lctsai.math@gmail.com}

\maketitle

\begin{abstract}
We prove that under a particular \emph{weak scaling},
the 4-parameter interacting particle system introduced by Corwin and Petrov \cite{corwin15}
converges to the \ac{KPZ} equation. 
This expands the relatively small number of systems 
for which \emph{weak universality} of the
\ac{KPZ} equation has been demonstrated.
\end{abstract}

\section{Introduction}
This paper demonstrates how the \ac{KPZ} equation \cite{kardar86} arises as a scaling limit of a 4-parameter interacting particle system introduced in \cite{corwin15} (called here the \ac{EP}) under fairly general choices of three parameters
($\nu\in [0,1)$, $\alpha>0$, $J\in\bbZ_{>0}$) and special tuning of the remaining paremeter ($q\to 1$). This system, through various specializations, and limit procedures includes all known integrable models in the \ac{KPZ} universality class. It is closely connected to the study of higher-spin vertex models within quantum integrable systems and hence enjoys a number of nice algebraic properties, some of which play important roles in our convergence proof.

The \ac{KPZ} equation is a paradigmatic continuum model for a randomly growing interface with local dynamics subject to smoothing, lateral growth and space-time noise (for more background, see the review \cite{corwin12}). Its spatial derivative solves the stochastic Burgers equation with conservative noise, and its exponential (Hopf-Cole transform) satisfies the \ac{SHE} with multiplicative white-noise. The connection to stochastic Burgers equation suggests a relation to interacting particle systems while the connection to the \ac{SHE} suggests a relation to directed polymer models (whose partition functions satisfy discrete versions of the \ac{SHE}).

The \ac{KPZ} equation is written as
\begin{align}\label{eq:KPZeq}
	\partial_{\tau}\Hlim(\tau,r)
	=
	\tfrac{1}{2}\delta \,\partial_r^2 \Hlim(\tau,r)
	+ \tfrac{1}{2} \kappa\, \big(\partial_r \Hlim(\tau,r)\big)^2
	+ \sqrt{D}\, \eta(\tau,r),
\end{align}
where $\eta$ is space-time white noise, $\delta,\kappa\in \bbR$, and $D>0$. Care is needed in making sense of the above equation, and the proper notion of solution is that of the {\it Hopf-Cole solution to the \ac{KPZ} equation} which is defined by setting $\Hlim(\tau,r) =\tfrac{\delta}{\kappa} \log \Zlim(\tau,r)$ where $\Zlim$ solves the well-posed \ac{SHE}
\begin{align}\label{eq:SHE}
	\partial_{\tau}\Zlim(\tau,r)
	=
	\tfrac{1}{2}\delta \,\partial_r^2 \Zlim(\tau,r)
	+ \tfrac{\kappa}{\delta} \sqrt{D} \,\Zlim(\tau,r) \eta(\tau,r).
\end{align}

To understand how a microscopic system might scale to the \ac{KPZ} equation, it helps to understand how the \ac{KPZ} equation itself scales. For real $b,z$ define $\Hlim^{\e}(\tau,r):= \e^b \Hlim (\e^{-z}\tau, \e^{-1}r)$. Then $\Hlim^\e$ satisfies the scaled equation
\begin{align*}
	\partial_{\tau}\Hlim^\e(\tau,r)
	=
	\e^{2-z}\tfrac{1}{2}\delta \,\partial_r^2 \Hlim^\e(\tau,r)
	+ \e^{2-z-b} \tfrac{1}{2} \kappa \,\big(\partial_r \Hlim^\e(\tau,r)\big)^2
	+ \e^{b-\frac{z}{2}+\frac{1}{2}} \sqrt{D} \, \eta(\tau,r).
\end{align*}
There exists no choice of $b,z$ for which the coefficients of the scaled equation remain unchanged. However, if one simultaneously changes the values of some of the $\delta,\kappa,D$ parameters as $\e$ changes, the \ac{KPZ} equation may scale to itself. If the \ac{KPZ} equation remains invariant under such a scaling, it stands to reason that a microscopic model with similar properties may converge to the equation under a similar type of scaling and tuning of parameters. Such scalings are generally called \emph{weak scalings} since they involve taking some of the $\delta,\kappa,D$ parameters to zero with $\e$. It is thus a goal to show the \emph{weak universality} of the \ac{KPZ} equation by demonstrating how under these scalings, the equation arises from a variety of different microscopic models. Weak universality should be distinguished from \emph{\ac{KPZ} universality} which holds that without any tuning of parameters, a variety of different systems will converge under the choice of $b=1/2$ and $z=3/2$ to a universal limit called the \emph{\ac{KPZ} fixed-point} \cite{Corwin15a}.

There are very few proved instances of weak universality of the \ac{KPZ} equation. The first result was in the context of the \ac{ASEP} \cite{bertini97} for near equilibrium initial condition (see also \cite{amir11} for step initial condition). The \ac{ASEP} result came under \emph{weak asymmetry} scaling through which $b=1/2, z=2$ and $\kappa\mapsto \e^{1/2} \kappa$ ($\delta$ and $D$ remain unscaled). This result was extended in \cite{dembo16} to certain non-nearest neighbor (and non-exactly solvable) exclusion processes. The only other weak universality result \cite{alberts14} was in the context of discrete directed polymers with arbitrary disordered distributions. This result came under \emph{weak noise} scaling through which $b=0, z=2$ and $D\mapsto \e D$ ($\delta$ and $\kappa$ remain unscaled).

Owing to the round-about Hopf-Cole definition of the \ac{KPZ} equation, in order to prove that a system converges to the \ac{KPZ} equation, one must transform it microscopically into an approximate \ac{SHE}. 
The work of \cite{Hairer} provides direct meaning to the \ac{KPZ} equation (though for $r$ on the torus, not the full real line). As of yet, this approach has not yielded weak universality results for the \ac{KPZ} equation.
The work of \cite{JaraGoncalves} defines an \emph{energy solution} 
for the \ac{KPZ} equation and shows tightness of a certain class of 
interacting particle systems at equilibrium
and that all limit points are energy solutions.
The recent work of \cite{gubinelli15} establishes the uniqueness
of equilibrium energy solution on the torus
(under the time-reversal symmetry assumption enjoyed by Markov processes at equilibrium).

The partition function for a directed polymer model naturally solves a microscopic \ac{SHE} with a simple noise. On the other hand, the ASEP result relies heavily on the existence of a microscopic Hopf-Cole transform (known as the G\"{a}rtner transform \cite{gaertner87}), and the resulting \ac{SHE} has a much more complicated noise. This renders the associated analysis quite challenging. The work of \cite{dembo16} also relies on an approximate form of the G\"{a}rtner transform.
%
%
Microscopic Hopf-Cole transforms are hard to come by. 
For the model considered herein, this transform is achieved in Proposition~\ref{prop:dSHE}. 
The first indication that such a transform should 
exist came from the Markov duality enjoyed by the model; see Remark~\ref{rmk:duality}
for more discussions.

The particular choice of weak scalings present in our result is new. It corresponds to the \ac{KPZ} equation scaling given by $b=1, z=3, \delta\mapsto \e \delta, \kappa \mapsto \e^2 \kappa$ and $D$ remaining unchanged. In terms of the scaling of the microscopic model, we have $b=1, z=3$ and $q=e^{-\e}$, while $\nu\in [0,1)$, $\alpha>0$ and $J\in\bbZ_{>0}$ remain fixed. One sees that microscopically, these choices of parameters corresponds to the above weak scaling. As there are many parameters at play, it is likely that there exist other weak scalings of the system which realized the same \ac{KPZ} equation limit.

There are a number of degenerations of the \ac{EP}, including the discrete time Bernoulli $q$-TASEP \cite{borodin13} and (through a limit transition) the continuous time $q$-TASEP \cite{BigMac,BCS}. Strictly speaking, our results do not immediately apply to the continuous time $q$-TASEP. 
The restriction on parameters $\nu\in [0,1)$, $\alpha>0$ and $J\in \bbZ_{>0}$ does not allow us to probe all of the degenerations of the system introduced in \cite{corwin15}. For instance, the stochastic six-vertex model \cite{BCG}(a discrete time version of ASEP) arises through a different choice of specialization, as does the $q$-Hahn TASEP \cite{CHahn}. These systems likewise enjoy dualities and one may hope to prove their weak universality. We leave this for future work.

\subsubsection*{Outline}
In Section \ref{sect:main} we introduce the 4-parameter particle system and then proceed to state our main results. These are stated in terms of the \ac{SHE} as Theorems~\ref{thm:main} and \ref{thm:step} (though Corollary \ref{cor:KPZ} provides the equivalent statements in terms of the \ac{KPZ} equation). Section \ref{sect:dSHE} provides the discrete Hopf-Cole transform satisfied by the system. Section \ref{sect:mom} provides moment estimates necessary to show tightness as $\e\to 0$. Section \ref{sect:mg} demonstrates how the limit points satisfy the martingale problem for the \ac{SHE}.

\subsubsection*{Acknowledgements}
IC wishes to thank Jeremy Quastel for discussions regarding his work on the convergence of $q$-TASEP to the \ac{KPZ} equation. 
We thank Yier Lin for careful reading of the manuscript and for pointing out a mistake (see Remark~\ref{rmk:error})
in an earlier version of this article (which is remedied in the present arXiv version, though not in the published version).
IC was partially supported by the NSF through DMS-1208998 as well as by the Clay Mathematics Institute through the Clay Research Fellowship, by the Institute Henri Poincare through the Poincare Chair, and by the Packard Foundation through a Packard Foundation Fellowship. LCT was partially supported by the NSF through DMS-0709248.
\section{Definition of the Model and Results}\label{sect:main}

We begin by recalling the definition of the \ac{EP}.
Let $ \N:=\bbZ_{\geq 0} $, $ \N^* := \N\cup\{\infty\} $ and $ \bbZ^*:=\bbZ\cup\{\infty\} $.
Define the space of right-finite particle configurations
\begin{align*}
	\Xsp_m
	:=
	\curBK{
		\vecx = (\infty =\ldots = x_{m-2}=x_{m-1}>x_{m}>x_{m+1}>\ldots) \in \bbZ^*	
	},
\end{align*}
where imaginary particles are placed at $ \infty $ for the convenience of notations,
and define the space of infinite particle configurations
$ \Xsp_{\infty} := \{\vecx=(\ldots>x_{-1}>x_{0}>x_{1}>\ldots)\in\bbZ^{\bbZ}\} $,
with the corresponding spaces of gap configurations
\begin{align*}
	\Gsp_m :=
	\curBK{ \vecg=(g_n) : g_n=\infty,\forall n< m; \quad g_n\in\N,\forall n\geq m },
	\quad
	\Gsp_\infty := \N^\bbZ.
\end{align*}
Fixing $ J\in\bbZ_{>0} $, we let $ \rmd(s) := s - \floorbk{s/J}J $,
or more explicitly,
$ (\rmd(0),\rmd(1),\ldots) = (0,1,\ldots,J-1,0,\ldots,J-1,\ldots) $.
Fixing $ q,\nu\in[0,1)$ and $\alpha>0$
we
let $ \alpha_j := \alpha q^{j} $
and $ \alpha(s) := \alpha_{\rmd(s)} $,
and equip our probability space
with independent Bernoulli random variables
\begin{align*}
	\B_n(s,g) \sim \Ber\BK{ \frac{\alpha(s)(1-q^g)}{1+\alpha(s)} },
	\quad
	\Bp_n(s,g) \sim \Ber\BK{ \frac{\alpha(s)+\nu q^g}{1+\alpha(s)} },
\end{align*}
indexed by $ (s,g,n)\in(\N,\N^*,\bbZ) $,
with the corresponding filtration
$ \scrF(t) := \sigma( \B_n(s,g), \Bp_n(s,g): (n,g)\in\N\times\N^*, s=0,\ldots,t-1 ) $.
Recall from \cite{corwin15} the following definition of the \ac{EP}.

\begin{definition}\label{def:EPseq}
Given $ \vecx(0)\in\Xsp_m $, a right-finite particle configuration,
we define an $ \Xsp_m $-valued Markov chain $ \{\vecy(s)\}_{s\in\N} $
by setting $ \vecy(0):=\vecx(0) $, and update $ \vecy(s) $ as follows.
We update $ \vecy(s) $ sequentially, starting from $ m $, by letting
\begin{align}\label{eq:seqn}
	y_m(s+1) = y_m(s) + \B_m(s,\infty),
\end{align}
and letting, for $ n>m $,
\begin{align}
	\label{eq:seqi}
	y_n(s+1) =
	\left\{ \begin{array}{l@{,}l}
		y_n(s) + \Bp_n(s,g_n(s))	&	\text{ if } y_{n-1}(s+1) > y_{n-1}(s),
		\\
		y_n(s) + \B_n(s,g_n(s))	&	\text{ if } y_{n-1}(s+1) = y_{n-1}(s),	
	\end{array}
	\right.
\end{align}
where $ g_n(s) := y_{n-1}(s) - y_{n}(s) -1 $ the $ n $-th gap of $ \vecy(s) $.
Namely, we move $ y_n(s) $ one step to the right with probability $ \frac{\alpha(s)}{1+\alpha(s)} $,
and subsequently, we move $ y_{n}(s) $ depending on how $ y_{n-1}(s) $ was updated:
if $ y_{n-1}(s) $ did not moved we then move $ y_n(s) $ one step to the right with probability 
$ \frac{\alpha(s)(1-q^{g_n(s)})}{1+\alpha(s)} $,
otherwise we move $ y_n(s) $ one step to the right 
with probability $ \frac{\alpha(s)+\nu q^{g_n(s)}}{1+\alpha(s)} $.

The \ac{EP} $ \{\vecx(t)\}_{t\in\N} $ is
then defined as the $ \Xsp_m $-valued Markov chain $ \vecx(t) := \vecy(Jt) $.
\end{definition}

\begin{remark}
The Markov chain $\vecy(t)$ was defined through a local sequential update of particles. Taking $J>1$ and considering $\vecx(t)$,
a priori one might think this property is lost. It was shown in \cite[Section 3]{corwin15} that, in fact, $\vecx(t)$ can be updated through a local sequential update (just like for $J=1$). In this case, each particle may move a distance between $0$ to $J$ sites to the right, and the jump probabilities depend on the length of the previous particle's jump as well as the length of the gap. The explicit form of this probability is somewhat more involved and given in \cite[Theorem~3.15]{corwin15}. In particular, given a gap $g$ and a jump of the previous particle by $h$, a particle jumps by $h'$ according to the probability given by $R^{(J)}_{\alpha} (g,h;g+h-h',h')$ where a concise formulas for  $R^{(J)}_\alpha$ is given in \cite[Theorem~3.15]{corwin15} and there is also a dependence on $q,\nu$ which is suppressed in the notation. For the purposes of this paper, it suffices to study the $\vecy(t)$ process and prove convergence of it to the \ac{KPZ} equation. It follows then immediately that the $\vecx(t)$ process likewise converges.
\end{remark}

The process introduced in Definition \ref{def:EPseq} is defined by
the \emph{sequential} update of \eqref{eq:seqn}--\eqref{eq:seqi},
which is inconvenient for our purpose.
We now recast the definition as a \emph{parallel} update,
and, as a byproduct, extend Definition~\ref{def:EPseq} to the space
$ \Xsp := \cup_{n\in\bbZ^*} \Xsp_n $
of possibly infinite particle configurations.
To this end we require the following lemma, which we prove in Section~\ref{sect:dSHE}.
\begin{lemma}\label{lem:move}
For any fixed $ \vecg \in (\N^*)^\bbZ $, $ s\in\N $, $ m\leq n\in\bbZ $,
letting
\begin{align}\label{eq:moveI}
	\moveI_{n,m}(s,\vecg)
	:=
	\Big(
		\prod_{n\geq i> m}
		\BK{ \Bp_{i}(s,g_{i}) - \B_{i}(s,g_{i}) }
	\Big)
	\B_m(s,g_m),
\end{align}
we have
\begin{align}\label{eq:move}
	\move_n(s,\vecg) := \sum_{m: n\geq m} \moveI_{n,m} \in\{0,1\},
\end{align}
where the series converges in $ L^k $ for all $ k\geq 1 $ and hence almost surely.
Further,
\begin{align}\label{eq:moveIter}
	\move_{n}(s,\vecg)
	=
	\move_{n-1}(s,\vecg) \Bp_n(s,g_n) + (1-\move_{n-1}(s,\vecg)) \B_n(s,g_n).
\end{align}
\end{lemma}

\begin{definition}\label{def:EPpar}
Fix $ m\in\N^* $ and $ \vecx(0)\in\Xsp_m $.
Letting $ \vecg(\vecy) := (y_{n-1}-y_{n}-1)_{n\in\bbZ} $ and
$ \vecmove(s,\vecg) := (K_n(s,\vecg))_{n\in\bbZ} \in \{0,1\}^\bbZ $ (by Lemma~\ref{lem:move}),
we define a stochastic map
\begin{align}\label{eq:trans}
	\Tran(s):	\Xsp_m \longrightarrow \Xsp_m,
	\quad
	\vecy \longmapsto \vecy + \vecmove(s,\vecg(\vecy)),
\end{align}
and define the $ \Xsp_m $-valued Markov chain $ \{\vecx(t)\} $ and $ \{\vecy(t)\} $
by letting
$ \vecy(s) := \Tran(s-1) \circ \Tran(s-2) \circ \cdots \circ \Tran(0)(\vecx(0)) $
and $ \vecx(t) := \vecy(tJ) $.
\end{definition}
\begin{remark}\label{rmk:EPequiv}
Under the map $ \Tran(s) $, we have that
$
	\move_{n}(s,\vecg(\vecy(s))) = \ind_\curBK{y_{n}(s+1)>y_n(s)},
$
whereby \eqref{eq:moveIter} becomes
\begin{align*}
	\ind_\curBK{ y_{n}(s+1)>y_n(s) }
	=
	\ind_\curBK{ y_{n-1}(s+1)>y_{n-1}(s) } \Bp_n(s,g_n(s))
	+
	\ind_\curBK{ y_{n-1}(s+1)=y_{n-1}(s) } \B_n(s,g_n(s)).
\end{align*}
This is equivalent to \eqref{eq:seqi}, which reduces to \eqref{eq:seqn} when $ g_{m-1}=\infty $.
It is thus easy to see that 
Definition~\ref{def:EPpar} is equivalent to Definition~\ref{def:EPseq}
when restricting to $ \vecx(0)\in\Xsp_m $, $ m\in\N $.
\end{remark}

Our main result is the convergence to the \ac{SHE} of a certain
exponential transform of the process $ \vecy(t) $.
Recall that we say a process $ \Zlim $ on $ \bbR_+\times\bbR $ is a mild
solution of the \ac{SHE} starting from the initial condition $ \Zic $ if
\begin{align}\label{eq:SHEmild}
	\Zlim(\tau,r)
	= \int_{\bbR} P_\tau(r-r') \Zic(r') dr'
		+ \int_0^{\tau} \int_{\bbR} P_{\tau-\tau'}(r-r') \Zlim(\tau',r') \eta(dr',d\tau'),
\end{align}
where $ P_\tau(r) := \exp[-r^2/(2\tau)] (2\pi\tau)^{-1/2} $ denotes the standard heat kernel,
and $\eta$ denotes the space-time white noise.
For the existence, uniqueness, continuity, and positivity of solutions of \eqref{eq:SHEmild},
see \cite[Proposition~2.5]{corwin12}.

The key step of showing the convergence to the \ac{SHE}
is finding a discrete \ac{SHE}.
To state it, we fix a parameter $ \dens\in(0,1) $,
measuring the limiting density (see Remark~\ref{rmk:dens}),
and set
\begin{align}
	&
	\label{eq:abconst}	
	\const := \frac{1-\dens}{1-\nu \dens},
	\quad\quad
	a_j := \frac{\alpha_j \const}{1+\alpha_j \const},
	\quad\quad
	b := \frac{\const}{1-\const},
	\quad\quad
	b' := \frac{\nu\const}{1-\nu\const},	
\\	
	\label{eq:driftt}
	&
	\drift(t) := (a_{\rmd(t)}-a_{\rmd(t)+1})(b-b')^{-1},
\\
	&
	\label{eq:centt}
	\cent(t) := \frac{1+\alpha(t) \const}{1+q\alpha(t) \const},
\\
	&
	\notag
	\Drift(t) := \sum_{s=0}^{t-1} \drift(s),
	\quad
	\Cent(t) := \prod_{s=0}^{t-1} \cent(s),&
\end{align}
with the conventions $ \Drift(0) := 0 $ and $ \Cent(0) := 1 $.
Letting $ Q_n(t) := q^{y_n(t)+n} $ denote the one-particle duality function
(see \cite{corwin15} for the definition of duality functions),
we define the exponential transform
\begin{align}\label{eq:Z}
	Z(t,\xi) := \Cent(t) \dens^{\xi+\Drift(t)} Q_{\xi+\Drift(t)}(t),
\end{align}
for $ t\in\N$ and $\xi \in \Xi(t) := (\bbZ-\Drift(t)) $.
The discrete \ac{SHE} is expressed in terms of
a certain random walk $ \RW(0)+\ldots+\RW(t-1) $ on $ \bbR $.
Here, $\RW(s)\in(\N - \drift(s))$, $ s\in\N $,
are independent random variables introduced in \eqref{eq:RW},
with zero mean and variance as in \eqref{eq:RWvar}.
Let $ \Xi(t_2,t_1) := \N+(\Drift(t_1)-\Drift(t_2)) $.
For $ t_1\leq t_2 $, $ \zeta\in\Xi(t_2,t_1) $,
let
\begin{align}
	\label{eq:hk}
	\hk(t_2,t_1,\zeta) := \Pro(\RW(t_1)+\RW(t_1+1)+\ldots+\RW(t_2-1)=\zeta)
\end{align}
denotes the corresponding semigroup.
We use the shorthand notations
$
	[\hk(t_2,t_1)*f(t_1)](\xi):=
	\sum_{\zeta\in\Xi(t_1)} \hk(t_2,t_1,\xi-\zeta) f(t_1,\zeta)
$
to denote convolution.
Let $ \moveCent_{n}(s) := \move_n(s) - \Ex(\move_n(s)|\scrF(s)) $ and let
\begin{align}\label{eq:mgg}
	\mgg(t,\xi) := \cent(t) (q-1) \moveCent_{\xi+\Drift(t)}(t),
\end{align}
representing the discrete analog of the space-time white noise.

\begin{proposition}\label{prop:dSHE}
For all $ t_1\leq t_2\in\N $ and $ \xi\in\Xi(t_2) $, we have
the following discrete \ac{SHE}
\begin{align}\label{eq:dSHE}
	Z(t_2,\xi)
	=
	\Zdr(t_2,t_1,\xi)
	+
	\Zmg(t_2,t_1,\xi),
\end{align}
where
\begin{align}
	&
	\label{eq:Zdr}
	\Zdr(t_2,t_1,\xi) := [\hk(t_2,t_1)*Z(t_1) ](\xi),
\\
	&
	\label{eq:Zmg}
	\Zmg(t_2,t_1,\xi)
	:=
	\sum_{s=t_1}^{t_2-1} [\hk(t_2,s+1)*(Z(s)W(s)) ](\xi+\drift(s)).
\end{align}
Further, for all $ \xi_1,\xi_2\in\Xi(t) $,
\begin{align}\label{eq:qv}
\begin{split}
	&
	Z(t,\xi_1) Z(t,\xi_2)
	\Ex \big[
		\mgg(t,\xi_1) \mgg(t,\xi_2) \big| \scrF(t)
		\big]
\\
	&
	\quad
	=
	\BK{\frac{(\nu+\alpha(t))\dens}{1+\alpha(t)}}^{|\xi_1-\xi_2|}
	\qv_1(t,\xi_1\wedge\xi_2) \qv_2(t,\xi_1\wedge\xi_2),
\end{split}
\end{align}
where
\begin{align}
	&
	\label{eq:qv1}
	\qv_1(t,\xi) :=
	q \cent(t) Z(t,\xi) - [\hk(t+1,t)*Z(t)](\xi-\drift(t)),
\\
	&
	\label{eq:qv2}
	\qv_2(t,\xi) :=
	- \cent(t)  Z(t,\xi) + \big[\hk(t+1,t)*Z(t)](\xi-\drift(t)).
\end{align}
\end{proposition}

\begin{remark}\label{rmk:duality}
The first indication that 
a microscopic Hopf-Cole transform as in Proposition~\ref{prop:dSHE} should exist 
came from the $ k=1 $ version of the Markov duality enjoyed by the model, 
given in \cite[Theorem~2.19]{corwin15}.
This result shows that $ \Ex(q^{y_n(t)+n}) $ 
satisfies the Kolmogorov backward equation in the $ t $ and $ n $ variables,
more explicitly 
\begin{align*}
	\Ex(q^{y_n(t+1)+n}) = \sum\nolimits_{m\in\bbZ} \hkp(t+1,t,n-m) \Ex(q^{y_{m}(t)+m}).
\end{align*}
Here $ \hkp(t+1,t,m) $ is the transition probability 
of a certain (time inhomogeneous) random walk $ \{X'(t)\}_{t\in\N} $, 
defined as in \eqref{eq:RWt},
which corresponds to the one-particle version of the 
Higher Spin Zero Range Process, defined in \cite[Definition 2.6]{corwin15}.
The existence of a nice martingale as in \eqref{eq:qv}
and finding the correct centering and tiling as in \eqref{eq:Z}
require further work given here in Proposition~\ref{prop:dSHE}.
\end{remark}

Proceeding to our main result,
we consider the weak noise scaling $ q=\qe := e^{-\e} $, $ \e\to 0 $.
Hereafter throughout the paper,
we \emph{fix} $\alpha>0$, $\nu\in[0,1)$, and $\dens\in (0,1)$,
and scale only the parameter $ \qe \to 1 $.
To indicate this scaling, we denote \emph{parameters} such as $ \alpha_j $ and $ \alpha(s) $
by $ \alphaej{j} $ and $ \alphae(s) $,
but for \emph{processes} such as $ \vecx(t) $, $ \B_n(s,g) $,
we often omit the dependence on $ \e $ to simplify notations.
Under this scaling, to the first order \eqref{eq:driftt}--\eqref{eq:centt} read
\begin{align}
	&
	\label{eq:driftte}
	\drifte(t) =
	\e \alpha\const (1+\alpha\const)^{-2}(b-b')^{-1} + O(\e^{2}),
\\
	&
	\label{eq:centte}
	\cente(t) = 1 + \e \alpha \const (1+\alpha\const)^{-1} + O(\e^{2}).
\end{align}
Let $ \xcent := (b-b')^{-1} $,
\begin{align}
	\label{eq:tcent}
	\tcente
	&:= \sqBK{ (a_{0})^2 -(\aej{J})^2 + (a_0-\aej{J})(b+b') }^{-1}
\\
	\label{eq:tcentApprox}	
	&=
	(1+\alpha\const)^2(J\alpha\const)^{-1}(2 a_0+b+b')^{-1} + O(\e).
\end{align}
Note that here $ a_0 = \alpha\const(1+\alpha\const)^{-1} $ is independent of $ \e $.
We extend the process $ Z(t,\xi) $, defined for $ t\in\N $ and $ \xi\in\Xi(t) $,
to a continuous process on $ \bbR_+\times\bbR $
by first linearly interpolate in $ \xi $ and then linearly interpolate in $ t $,
and then we introduce the scaled process
\begin{align}\label{eq:Ze}
	\Ze(\tau,r) := Z(\e^{-3}\tcente J \tau, \e^{-1} \xcent r),
\end{align}
or, equivalently $ \Ze(\tau,r) = \exp(H_\e(\tau,r)) $, where
\begin{align}\label{eq:He}
	H_\e(\tau,r)
	:=&
	-\e  y_{\ne(\tau,r)}(\te(\tau))
	+ \BK{\log\rho -\e} \ne(\tau,r)
	+ \log\Cente(\te(\tau)),
\end{align}
$ \te(\tau) :=  \e^{-3}\tcente J\tau $ and $ \ne(\tau,r) := \e^{-1}\xcent r + \Drifte(\te(\tau)) $.
Following \cite{bertini97}, we consider \emph{near equilibrium} initial conditions:
\begin{definition}\label{def:neareq}
Let $ \Ze(0,\xi) $ be the exponential transform (given as in \eqref{eq:Ze})
associated with $ \{\vecxe(0)\}_\e\subset\Xsp $.
We say $ \{\vecxe(0)\}_\e\subset\Xsp $
is \emph{near equilibrium} if, given any $ k\in\bbZ_{>0} $ and $ v\in(0,1/2) $,
there exists $ u=u(k,v), C=C(k,v)<\infty $ such that
\begin{align}
	&
	\label{eq:neareq:mom}
	\VertBK{\Ze(0,\xi)}_{k} := \BK{ \Ex \BK{\Ze(0,\xi)^k} }^{1/k}
	\leq C e^{u|\xi|},
\\
	&
	\label{eq:neareq:momx}
	\VertBK{\Ze(0,\xi)-\Ze(0,\xi')}_k \leq C|\xi-\xi'|^v e^{u(|\xi|+|\xi'|)},	
\end{align}
for all $ \xi,\xi'\in \e(\xcent)^{-1}\bbZ $ and $ \e>0 $ small enough.
\end{definition}
\noindent
Hereafter we endow the spaces $ C(\bbR) $, $ C(\bbR_{+}\times\bbR) $ and $ C((0,\infty)\times\bbR) $
the topology of uniform convergences on compact subsets,
and use $ \Rightarrow $ to denote weak convergence of probability laws.
The following is our main result.
\begin{theorem}\label{thm:main}
Let $ \Zlim $ be the unique $ C(\bbR_+\times\bbR) $-valued solution of \ac{SHE}
starting from a $ C(\bbR) $-valued process $ \Zic $,
and let $ \Ze(\tau,r)\in C(\bbR_+\times\bbR) $ be as in \eqref{eq:Ze},
with some near equilibrium initial condition $ \{\vecxe(0)\}_\e $.
If $ \Ze(0,\Cdot) \Rightarrow \Zic(\Cdot)  $,
then
\begin{align*}
	\Ze(\Cdot,\Cdot) \Rightarrow \Zlim(\Cdot,\Cdot)
	\text{ on }
	C(\bbR_+\times\bbR),
	\quad
	\text{ as } \e \to 0.
\end{align*}
\end{theorem}

Definition~\ref{def:neareq} (and therefore Theorem~\ref{thm:main})
leaves out an important initial condition,
i.e.\ the step initial condition: $ x_n(0) := {-n} $ for $ n\in\N $ and $ x_{n}=\infty $ for $ n\in\bbZ_{<0} $.
Following \cite{amir11}, we generalize Theorem~\ref{thm:main} to the following:

\begin{theorem}\label{thm:step}
Let $ \Zdlim(\Cdot,\Cdot) $ be the unique solution of \ac{SHE}
starting from the delta measure $ \delta(\Cdot) $,
let $ \{\vecy(t)\}_{t} \in \Xsp_0 $ be the process starting from the step initial condition,
and let $ \Zde(\tau,r) := \e^{-1}(1-\rho)\xcent\Ze(\tau,r) $.
Then
\begin{align*}
	\Zde(\Cdot,\Cdot) \Rightarrow \Zdlim(\Cdot,\Cdot)
	\text{ on }
	C((0,\infty)\times\bbR),
	\quad
	\text{ as } \e \to 0.
\end{align*}
\end{theorem}

\begin{remark}\label{rmk:dens}
By Theorem~\ref{thm:main} we have that $ H_\e(\tau,r)-H_\e(\tau,r+\e/\xcent) \to 0 $ in probability.
Plugging this in \eqref{eq:He}, a posteriori we find that,
$ y_{\ne(\tau,r)}(\te(\tau)) - y_{\ne(\tau,r)+1}(\te(\tau)) \approx \e^{-1} \log(1/\dens) $,
or equivalently the limiting density is $ \e/\log(1/\dens) $.
\end{remark}

Hereafter we adapt the convention that
$ m,n,i,j,k\in\bbZ $; $ s,t\in\N $; $ \tau,\tau'\in\bbR_+ $; and $ r\in\bbR $.
To simplify notation,
we let $ \vecg(t) := \vecg(\vecy(t)) $, $ \B_n(t):=\B_n(t,\vecg(t)) $, $ \Bp_n(t):=\Bp_n(t,\vecg(t)) $,
$ \move_n(t):= \move_n(t,\vecg(t)) $ and $ \moveI_{n,m}(t):= \moveI_{n,m}(t,\vecg(t)) $,
with the consensus that an underlying process $ \vecy(t) $ has been fixed.
We will specify explicitly when a result applies only for near equilibrium initial
conditions or the step initial condition,
and without specification the result holds for any initial condition $ \vecx(0)\in\Xsp $.

%
%
%

\begin{proof}[Proof of Theorem~\ref{thm:main}]
This is an immediate consequence of the following propositions,
which we establish in Section~\ref{sect:mom} and \ref{sect:mg}, respectively.
\begin{proposition}\label{prop:tight}
For near equilibrium initial conditions,
the collection of processes $ \{\Ze\}_\e $ is tight in $ C(\bbR_+\times\bbR) $.
\end{proposition}
\begin{proposition}\label{prop:unique}
For near equilibrium initial conditions,
any limiting point $ \Zlim$ of $ \{\Ze\}_\e $ solves the \ac{SHE}.
\end{proposition}
\end{proof}
%
%
\begin{proof}[Proof of Theorem~\ref{thm:step}]
We let $ \Zd(\tau,r) := \xcent \e^{-1}(1-\dens) Z(\tau,r) $
so that $ \Zde(\tau,r) = $\\ $ \Zd(\e^{-3}\tcente \tau, \e^{-1}\xcent r) $.
The prefactor of $ \Zd(\tau,r) $ is choose so that
\begin{align}\label{eq:Zdmass}
	\e\xcent^{-1} \sum_{\xi\in\Xi(0)} \Zd(0,\xi) = 1.
\end{align}
Further, using the exponential decay (in $ |\xi| $) of $ \Zd(0,\xi) $,
one easily obtains $ \Zde(0,\Cdot) \Rightarrow \delta(\Cdot) $.
With this and Theorem~\ref{thm:main},
following the argument of \cite[Section 3]{amir11},
Theorem~\ref{thm:step} is an immediate consequence of
the following moment estimates of $ \Zd(\tau,r) $, which we establish in Section~\ref{sect:mom}.
\begin{proposition}\label{prop:stepEst}
For the step initial condition,
for any $ T>0 $, $ k\geq 1 $ and $ v\in(0,1/2) $,
there exists $ C=C(T,k,v)<\infty $ such that
\begin{align}
	&
	\label{eq:momd}
	\Vertbk{ \Zd(\tau,r) }_{2k} \leq C (\e^{3}\tau)^{-1/2},
\\
	&
	\label{eq:momdx}
	\Vertbk{ \Zd(\tau,r) - \Zdlim(\tau,r') }_{2k} \leq C (\e|r-r'|)^{v} (\e^3\tau)^{-(1+v)/2},
\end{align}
for all $ \tau\in(0,\e^{-3}T] $ and $ r,r'\in\bbR $.
\end{proposition}
\end{proof}
With $ \vecx(t) = \vecy(tJ) $,
from Theorems~\ref{thm:main}--\ref{thm:step}
we immediately obtain the following corollary
on the convergence of $ \vecx(t) $.
More precisely, 
letting $ \drifte := \sum_{s=0}^{J-1} \drifte(s) $
and $ \cente := \sum_{s=0}^{J-1} \cente(s) $,
we define
\begin{align*}
	\Hx(\tau,r) :=
	-\e  x_{\e^{-1} \xcent r + \e^{-3}\drifte \tcente\tau }(\e^{-3}\tcente\tau)
	+ \BK{\log\rho -\e} (\e^{-1} \xcent r + \e^{-3}\drifte \tcente\tau )
	+ \log(\e^{-3} \cente \tcente \tau),
\end{align*}
(which is defined on $ \bbR_+\times\bbR $ by the aforementioned linear interpolation).
With $ H_\e(\tau,r) $ as in \eqref{eq:He},
we have that $ \Hx(\tau,r) = H_\e(\tau,r) $, for all $ \tau\in \e^3 {\tcente}^{-1}\N $ and $ r\in\bbR $.
From this, Theorems~\ref{thm:main}--\ref{thm:step} immediately imply

\begin{corollary}\label{cor:KPZ}
\begin{enumerate}[label=(\alph*)]
	\item[]
	\item
	Let $ \Zlim(\tau,r) $ and $ \Zic(r) $ be as in Theorem~\ref{thm:main}
	so that $ \Hlim(\tau,r) := \log \Zlim(\tau,r) $
	is the unique solution of the \ac{KPZ} equation starting from $ \log \Zic $,
	and let $ \{\vecx^\e(0)\}_\e $ be a collection of near equilibrium initial conditions.
	If $ \Ze(0,\Cdot) \Rightarrow \Zic(\Cdot) $,
	we have
	\begin{align}\label{eq:KPZcnvg}
		\Hx \Rightarrow \Hlim \text{ on } C((0,\infty)\times\bbR),
		\quad
		\text{ as } \e\to 0. 
	\end{align}
	\item	
	Let $ \Zdlim(\tau,r) $ be as in Theorem~\ref{thm:step},
	let $ \{\vecx^\e(t)\}_\e $ be started from the step initial condition,
	and let $ \Hdx(\tau,x) := \Hx(\tau,x) + \log(\e^{-1}(1-\dens)\xcent) $.
	We have
	\begin{align*}
		\Hdx \Rightarrow \Hdlim \text{ on } C((0,\infty)\times\bbR),
		\quad
		\text{ as } \e\to 0. 
	\end{align*}
\end{enumerate}
\end{corollary}

\begin{remark}
\begin{enumerate}[label=(\alph*)]
\item[]
\item
In \eqref{eq:KPZcnvg} the convergence does not include $ \tau=0 $ as
we do not assume $ \Zic(r)>0 $.
\item
From Theorems~\ref{thm:main} and \ref{thm:step},
one also easily obtains corresponding convergence results for $ \Ze(\tau J,r) $,
the centered scaled exponential transform of $ \vecx(t) $,
but we do not state the results explicitly here.
\end{enumerate}
\end{remark}

\section{Discrete \ac{SHE}, Proof of Proposition~\ref{prop:dSHE}}
\label{sect:dSHE}

\begin{proof}[Proof of Lemma~\ref{lem:move}]
Fixing $ s\in\N $ and $ \vecg\in(\N^*)^\infty $,
we let $ \move_n $ and $ \moveI_{n,i} $
denote $ \move_n(s,\vecg) $ and $ \moveI_{n,i}(s,\vecg) $, respectively,
and let $ \move_{n,i} := \sum_{n \geq i' \geq i} \moveI_{n,i'} $
denote the $ i $-th partial sum of \eqref{eq:move}.
With $ \B_{k}(s,g) $ and $ \Bp_{k}(s,g) $ defined as in the preceding,
we have
\begin{align*}
	&
	\Ex\absBK{ \Bp_{k}(s,g) - \B_{k}(s,g) }	\leq 1 - \frac{1}{1+\alpha}\frac{1-\nu}{1+\alpha} < 1,&
	&
	\Ex \B_k(s,g) \leq \frac{\alpha}{1+\alpha} < 1.
\end{align*}
Consequently, $ \move_{n,i} \to \move_n $ (as $ i\to-\infty $)
in $ L^k $ for all $ k\geq 1 $.

To show $ \move_n\in\{0,1\} $,
first we use the identity $ \moveI_{n,i} = (\Bp_n-\B_n)\moveI_{n-1,i} $
(which follows from \eqref{eq:moveI}) to obtain
\begin{align}\label{eq:moveIter:}
	\move_{n,i-1}(s,\vecg)
	=
	\move_{n-1,i-1}(s,\vecg) \Bp_n(s,g_n) + (1-\move_{n-1,i-1}(s,\vecg)) \B_n(s,g_n).
\end{align}
We now show that, in fact, $ \move_{n,i}\in\{0,1\} $ for all $ n\geq i $.
Indeed, $ K_{n,n} = B_n \in\{0,1\} $.
The general case then follows by induction on $ n-i\in\N $ using \eqref{eq:moveIter:}.
Consequently, $ \move_{n,i} \to \move_n \in\{0,1\} $.

The identity \eqref{eq:moveIter} follows directly by letting $ i\to-\infty $
in \eqref{eq:moveIter:}.
\end{proof}

Turning to proving Proposition~\ref{prop:dSHE},
as this does not involve the scaling $ \e\to 0 $,
throughout this section we \emph{suppress} the dependence of parameters on $ \e $.
We begin by deriving an equation for $ Q_n(t) $.
Consider the time-inhomogeneous random walk
$ X'(t+1) := \RWp(0) + \RWp(1) + \ldots + \RWp(t) $,
where $ \RWp(s)$, $ s\in\N $, are $ \N $-valued, independent, with distribution
\begin{align}\label{eq:RWt}
	\Pro(\RWp(t) =n )
	:=
	\hkp(t+1,t,n)
	:=
	\left\{\begin{array}{l@{,}l}
			\frac{\alpha(t)(1-q)}{1+\alpha(t)}
			\BK{ \frac{\nu+\alpha(t)}{1+\alpha(t)} }^{n-1}
			\BK{ 1 - \frac{\nu+\alpha(t)}{1+\alpha(t)} }
			&
			\text{ for } n>0,
			\\
			1 - \frac{\alpha(t)(1-q)}{1+\alpha(t)}
			&
			\text{ for } n=0,
			\\
			0
			&
			\text{ otherwise}.
	\end{array}\right.
\end{align}
Let $ [\hkp(t+1,t)*Q(t)]_n := \sum_{m\in\bbZ} \hkp(t+1,t,n-m) Q_{m}(t) $ denote convolution.

\begin{proposition}
For any $ t\in\N $ and $ n \in\bbZ$, we have
\begin{align}
	\label{eq:duality}
	Q_{n}(t+1) = [\hkp(t+1,t) * Q(t)]_n + Q_n(t) \mgg'_n(t),
\end{align}
where $ \mggp_n(t) := (q-1) \moveCent_n(t) $.
Further, for any $n_1, n_2\in\bbZ $,
\begin{align}\label{eq:qvraw}
\begin{split}
	&
	Q_{n_1}(t)Q_{n_2}(t) \Ex\BK{ \mggp_{n_1}(t)\mggp_{n_2}(t) \middle| \scrF(t) }
\\
	&
	\quad
	=
	\BK{\frac{\nu+\alpha(t)}{1+\alpha(t)}}^{|n_1-n_2|}
	\qvp_1(t,n_1\wedge n_2) \qvp_2(t,n_1\wedge n_2),
\end{split}
\end{align}
where
$ \qvp_{1}(t,n) := qQ_{n}(t) - [\hkp(t+1,t) * Q(t)]_{n} $
and
$ \qvp_{2}(t,n) := [\hkp(t+1,t) * Q(t)]_{n} - Q_{n}(t) $.
\end{proposition}

\begin{proof}
Fixing $ t\in\N $, to simplify notation we let $ \Ex'(\Cdot) $ denote $ \Ex(\Cdot|\scrF(t)) $.
We begin by proving \eqref{eq:duality}.
With $ Q_n(t) := q^{y_n(t)+n} $,
a generic jump $ y_n(t) \mapsto y_{n}(t)+1 $ of particles decreases $ Q_n(t) $
by $ (1-q) Q_n(t) $.
Consequently,
\begin{align*}
	Q_n(t+1) - Q_n(t)
	=  (q-1) Q_n(t) \move_n(t)
	=  (q-1) Q_n(t) \Ex'(\move_n(t)) + Q_n(t) \mggp_n(t).
\end{align*}
With $ \move_n(t) $ as in \eqref{eq:move},
we have
\begin{align}\label{eq:moveEx}
	\Ex'(\move_n(t))
	=
	\sum_{m:n\geq m}
	\frac{\nu+\alpha(t)}{1+\alpha(t)} q^{g_n(t)} \cdots \frac{\nu+\alpha(t)}{1+\alpha(t)} q^{g_{m+1}(t)}
	\frac{\alpha(t)}{1+\alpha(t)} (1-q^{g_m(t)}).
\end{align}
Multiplying both sides by $ (q-1)Q_n(t) $,
and then using the readily verify identity
\begin{align}\label{eq:Qshift}
	Q_n(t) q^{g_n(t)+g_{n-1}(t)+\ldots +g_{m'+1}(t)} = Q_{m'}(t),
\end{align}
we obtain
\begin{align}\label{eq:moveQ}
	(q-1)Q_n(t) \Ex'(\move_n(t)) = [\hkp(t+1,t)*Q(t)]_n - Q_n(t),
\end{align}
whereby \eqref{eq:duality} follows.

Turning to \eqref{eq:qvraw},
without lost of generality we assume $ n_1\geq n_2 $.
With $ \mggp_n(t) $ defined as in the proceeding,
we have
$
	\Ex'(\mggp_{n_1}(t)\mggp_{n_2}(t)) = (q-1)^2 \Cov'(\move_{n_1}(t),\move_{n_2}(t))
$,
where
$
	\Cov'(\move_{n_1}(t),\move_{n_2}(t))
	:= \Ex'(\move_{n_1}(t)\move_{n_2}(t)) - \Ex'(\move_{n_1}(t)) \Ex'(\move_{n_2}(t))
$.
Letting  $ \moveIp_{n_1,n_2}(t) :=$\\ $\prod_{n_1\geq k > n_2} ( \Bp_k(t) - \B_k(t) ) $,
with $ \move_{n_1}(t) $ as in \eqref{eq:move},
we have
\begin{align*}
	\move_{n_1}(t)
	=
	\sum_{n_1 \geq m > n_2} \moveI_{n_1,m}(t) + \moveIp_{n_1,n_2}(t) \move_{n_2}(t),
\end{align*}
for all $ n_1 \geq n_2 $.
Multiply both sides by $ \move_{n_2}(t) $, using $ \move_{n_2}(t)^2 = \move_{n_2}(t) $,
and then take the expectation $ \Ex'(\Cdot) $ on both sides.
With $ \{\B_k(s),\Bp_k(s)\}_k $ being independent, we obtain that
\begin{align*}
	\Ex'(\move_{n_1}(t) \move_{n_2}(t))
	&=
	\Big( \sum_{n_1\geq m > n_2} \Ex'(\moveI_{n_1,m}(t)) + \Ex'(\moveIp_{n_1,n_2}(t)) \Big)
	\Ex'(\move_{n_2}(t)).
\end{align*}
Subtracting
$ \Ex'(\move_{n_1}(t)) \Ex'(\move_{n_2}(t)) = [\sum_{m:n_1\geq m} \Ex'(\moveI_{n_1,m}(t))] \Ex'(\move_{n_2}(t)) $
from the last expression yields
\begin{align*}
	\Cov'(\move_{n_1}(t) \move_{n_2}(t))
	=
	\Big( -\sum_{m:n_2\geq m} \Ex'( \moveI_{n_1,m}(t) ) + \Ex'( \moveIp_{n_1,n_2}(t) ) \Big)
	\Ex'(\move_{n_2}(t)).
\end{align*}
Further using $ \Ex'(\moveI_{n_1,m}(t)) = \Ex'(\moveI_{n_2,m}(t)) \Ex'(\moveIp_{n_1,n_2}(t)) $,
we arrive at
\begin{align}
	\label{eq:moveCov}
	\Cov'(\move_{n_1}(t) \move_{n_2}(t))
	=
	\Ex'( \moveIp_{n_1,n_2}(t) )
	\Big( -\Ex'(\move_{n_2}(t)) + 1 \Big) \Ex'(\move_{n_2}(t)).
\end{align}
With
$
	\Ex'( \moveIp_{n_1,n_2}(s) )
	= \BK{\frac{\nu+\alpha(t)}{1+\alpha(t)}}^{n_2-n_1}
	 q^{g_{n_1}(t)+\ldots+g_{n_2+1}(t)},
$
multiplying both sides of \eqref{eq:moveCov} by $ (q-1)^{2}Q_{n_1}(t) Q_{n_2}(t) $,
and then applying \eqref{eq:Qshift}--\eqref{eq:moveQ}, we conclude \eqref{eq:qvraw}.
\end{proof}

We next introduce a centering to $ \RWp(t) $.
Let
\begin{align}\label{eq:vart}
	\var(t) := (a_{\rmd(t)})^2-(a_{\rmd(t)+1})^2+(a_{\rmd(t)}-a_{\rmd(t)+1})(b+b').
\end{align}

\begin{lemma}
For any $ t\in\N $, we have
$
	\Ex \BK{ \cent(t)\rho^{\RWp(t)} } = 1,
$
so that
\begin{align}\label{eq:RW}
	\Pro\BK{ \RW(t) + \drift(t) = n } := \cent(t) \rho^{n} \Pro( \RWp(t)=n ),
	\quad
	n\in\N
\end{align}
defines an  $(\N -\drift(t))$-valued random variable.
Further, $ \Ex(\RW(t))=0 $ and
\begin{align}\label{eq:RWvar}
	\Ex(\RW(t)^2) = \xcent^2 \var(t).
\end{align}
\end{lemma}

\begin{proof}
Fixing $ t\in\N $,
we consider the function
\begin{align}
	\notag
	f(x) &:= \Ex(\cent(t) x^{\RWp(t)})
\\
	\notag
	&=
	\cent(t)
	\sqBK{
		\BK{ 1 - \frac{(1-q)\alpha(t)}{1+\alpha(t)} }
		+
		\sum_{i=1}^\infty x^i
		\frac{(1-q)\alpha(t)}{1+\alpha(t)}
		\BK{ \frac{\nu+\alpha(t)}{1+\alpha(t)} }^{i-1}
		\frac{1-\nu}{1+\alpha(t)}
		}
\\
	\label{eq:RWf}
	&=
	\cent(t)
	\frac{1-\nu x +q\alpha(t) - q\alpha(t) x}{1-\nu x+\alpha(t)-\alpha(t) x}.
\end{align}
With $ \cent(t) $ defined as in \eqref{eq:centt},
specializing \eqref{eq:RWf} at $ x=\dens $ we obtain $ f(\dens)=1 $,
thereby concluding $ \Ex(\cent(t)\dens^{\RWp(t)})=1 $.
Next, differentiating $ f(x) $ yields
\begin{align}
	&
	\label{eq:RWfp}
	\BK{ x \tfrac{d}{dx} f } (\dens) = \Ex\BK{ \cent(t) \dens^{\RWp(t)} \RWp(t) } = \Ex(\RW(t)+\drift(t))
\\
	&
	\label{eq:RWfpp}
	\BK{ x \tfrac{d}{dx} \BK{ x \tfrac{d}{dx} f} } (\dens) =  \Ex\BK{ \cent(t) \dens^{\RWp(t)} \RWp(t)^2 }
	 = \Ex((\RW(t)+\drift(t))^2).
\end{align}
Plugging \eqref{eq:RWf} into the l.h.s.\ of \eqref{eq:RWfp}--\eqref{eq:RWfpp}
and specializing at $ x=\dens $,
after some tedious but straightforward calculations,
one obtains $ ( x \frac{df}{dx} ) (\dens) = \drift(t) $
and $ ( x \diff ( x \diff f) ) (\dens) = \drift(t)^2 + \xcent^2\var(t) $,
thereby concluding $ \Ex(\RW(t))=0 $ and \eqref{eq:RWvar}.
\end{proof}

\begin{proof}[Proof of Proposition~\ref{prop:dSHE}]
With $ [\hkp(t+1,t)*Q(t)]_n = \Ex(Q_{n-\RWp(t)}(t)) $
and $ [\hk(t+1,t)*Z(t)](\xi) = \Ex( Z(t,\xi-\RW(t)) ) $,
we have the readily verified identity
\begin{align}\label{eq:hkp2hk}
	\Cent(t+1) \dens^{\xi+\Drift(t+1)}
	[\hkp(t+1,t) * Q(t) ]_{\xi+\Drift(t+1)}
	=
	[\hk(t+1,t)*Z(t)](\xi),
\end{align}
for all $ \xi\in\Xi(t) $.
In \eqref{eq:duality}, we set $ n=\xi+\Drift(t+1) $,
and multiply both sides by $ \Cent(t+1) \dens^{\xi+\Drift(t+1)} $.
Using \eqref{eq:Z} and \eqref{eq:hkp2hk},
we obtain
\begin{align}\label{eq:dSHE:}
	Z(t+1,\xi)
	= \big[ \hk(t+1,t) * Z(t) \big](\xi)
	 + Z(t,\xi+\drift(t)) W(t,\xi+\drift(t)).
\end{align}
Iterating this equation from $ t=t_2-1 $ until $ t=t_1 $,
we thus conclude \eqref{eq:dSHE}.

To derive \eqref{eq:qv},
in \eqref{eq:qvraw},
we set $ n_1=\xi_1+\Drift(t) $ and $ n_2=\xi_2+\Drift(t) $,
and multiply both sides by $ \Cent(t+1)^2 \dens^{\xi_1+\Drift(t)} \dens^{\xi_2+\Drift(t)} $.
Using \eqref{eq:Z} and \eqref{eq:hkp2hk} to express the resulting equation
in terms of $ Z(t,\Cdot) $ and $ \hk(t+1,t,\Cdot) $, we thus conclude \eqref{eq:qv}.
\end{proof}

\section{Moment Estimates: Proof of Propositions~\ref{prop:tight} and \ref{prop:stepEst}}
\label{sect:mom}

Hereafter we let $ C(u_1,u_2,\ldots) $ denote a generic finite positive constant
that depends only on the designated variables $ u_1,u_2,\ldots $
and possibly on $\alpha>0$, $\nu\in [0,1)$ and $\dens\in (0,1)$, which are fixed throughout the paper.

\begin{lemma}\label{lem:Taylor}
The function $ \phi(x,\e;t) := \Ex(x^{\RW_\e(t)}) $
extends analytically in $ (x,\e) $ to a neighborhood of $ (1,0)\in\bbC^2 $,
with the Taylor expansion
\begin{align}\label{eq:RWTaylor}
	\phi(x,\e;t) = 1 + 2^{-1} \big( \partial_{xx\e} \phi(1,0;t) \big) \e(x-1)^2 + O(\e|x-1|^3),
\end{align}
and $ \partial_{xx\e} \phi(1,0;t) \in (0,\infty) $.
\end{lemma}

\begin{proof}
Since $ \RW_\e(t) $ is defined by $ \RWp_\e(t) $ as in \eqref{eq:RW}, we clearly have
\begin{align}\label{eq:phif}
	\phi(x,\e;t) = \Ex( \cente(t) \dens^{\RWp_\e(t)} x^{\RWp_\e(t)-\drifte(t)} ).
\end{align}
By \eqref{eq:RWf}, the function $ \Ex( \cente(t) x^{\RWp_\e(t)} ) $ is analytic
in $ (x,\e) $ within a neighborhood of $ (\rho,0) $,
whereby $ \phi(x,\e;t) $ is analytic within a neighborhood of $ (1,0) $.
To obtain the Taylor expansion \eqref{eq:RWTaylor},
we differentiate \eqref{eq:phif} in $ x $, and then specialized at $ x=1 $, thereby obtaining
\begin{align}
	\label{eq:f1pf}
	\partial_x \phi(1,\e;t) = \Ex\BK{ \cente(t) \dens^{\RWp_\e(t)} (\RWp_\e(t)-\drifte(t)) }
	&=
	\Ex(\RW_\e(t)) =0,
\\
	\label{eq:f1ppf}
	\partial_{xx} \phi(1,\e;t) = \Ex\BK{ \cente(t) \dens^{\RWp_\e(t)} (\RWp_\e(t)-\drifte(t))^2 }
	&=
	\Ex(\RW_\e(t)^2) = \xcent^2 \vare(t).
\end{align}
With $ \vare(t) $ defined as in \eqref{eq:vart}, we have
\begin{align}\label{eq:vartApprox}
	\vare(t) = \e \alpha\const(1+\alpha\const)^{-3} \BK{ 2\alpha\const + (b+b')(1+\alpha\const) } + O(\e^2).
\end{align}
From \eqref{eq:f1pf}--\eqref{eq:vartApprox}
we conclude that $ \partial^n_{x} \partial^m_\e\phi(1,0;t)=0 $, unless $ n\geq 2 $ and $ m\geq 1 $,
and that $ \partial_{xx\e}\phi(1,0;t) >0 $,
thereby obtaining \eqref{eq:RWTaylor}.
\end{proof}
Based on Lemma~\ref{lem:Taylor},
we proceed to estimating of the heat kernel.
\begin{proposition}\label{prop:hk}
Given any $ u,T> 0$ and $ v\in(0,1] $,
there exists $ C=C(T,u) $ such that
\begin{align}
	&
	\label{eq:hksum}
	\sum_{\zeta\in\Xi(t_2,t_1)}
	\hke(t_2,t_1,\zeta) e^{ u\e|\zeta| }	
	\leq C,
\\
	&
	\label{eq:hksumt}
	\sum_{\zeta\in\Xi(t_2,t_1)}
	|\zeta|^{v}\hke(t_2,t_1,\zeta) e^{ u\e|\zeta| }	
	\leq
	C (\e|t_2-t_1|)^{v/2},
\\
	&
	\label{eq:hkbdd}
	\hke(t_2,t_1,\xi)
	\leq
	C \e^{-1/2}(t_2-t_1+1)^{-1/2},
\\
	&
	\label{eq:hkx}
	\absBK{ \hke(t_2,t_1,\xi) - \hke(t_2,t_1,\xi') }
	\leq
	C \e^{-(1+v)/2} |\xi-\xi'|^{v} (t_2-t_1+1)^{-(1+v)/2},
\end{align}
for all $ t_1\leq t_2\in[0,\e^{-3}T]\cap\N $
and $ \xi,\xi'\in\Xi(t_2) $.
\end{proposition}

\begin{proof}
To prove \eqref{eq:hksum}, we consider
$ F_1(u') := e^{u'(\RW_\e(t_1)+\ldots+\RW_\e(t_2-1))} $.
With
$
	\Ex(F_1(u')) =$ \\$\sum_{\zeta\in\Xi(t_2,t_1)} \hke(t_2,t_1,\zeta) e^{u'\zeta}
$
and $ e^{\e u|\zeta|} \leq e^{\e u \zeta} + e^{-\e u \zeta} $,
to show \eqref{eq:hksum}, it suffices to bound the expression
$
	\Ex(F_1(u'))
	= \prod_{s=t_1}^{t_2-1} \phi(e^{u'},\e;s),
$
for $ u'=\pm u\e $.
This, by \eqref{eq:RWTaylor}, is bounded by $ [1+C\e(e^{u'}-1)^2]^{t_2-t_1} $.
With $ t_2-t_1 \leq \e^{-3}T $, the last expression is bounded by $ C=C(T,u) $,
from which we conclude \eqref{eq:hksum}.

Turning to showing \eqref{eq:hksumt},
we let $ F_2 := \RW_\e(s_1)+\ldots+\RW_\e(s_2-1) $.
Similar to the preceding, it suffices to bound the expression
\begin{align*}
	\sum_{\zeta\in\Xi(s_1,s_2)} |\zeta|^{v} \hke(s_1,s_2,\zeta) e^{u'\zeta}	
	=
	\Ex\BK{ F_1(u') |F_2|^{v}  }
	\leq
	\big\Vert F_1(u') \big\Vert_{2/(2-v)}
	\big\Vert |F_2|^{v} \big\Vert_{2/v},
\end{align*}
for $ u'=\pm u\e $, where we used H\"{o}lder's inequality in the last inequality.
With $ (F_1(u'))^{2/(2-v)} = F_1(2u'/(2-v)) $,
applying \eqref{eq:hksum} for $ u = 2u/(2-v) $
we obtain $ \Vertbk{ F_1(\pm u\e) }_{2/(2-v)} \leq C $.
As for $ F_2 $, with $ \Ex(\RW_\e(s))=0 $ and \eqref{eq:RWvar}, we have
\begin{align*}
	\big\Vert |F_2|^{v} \big\Vert_{2/v} = \sqBK{ \Ex(F_2)^2 }^{v/2}
	&=
	\sqBK{ \Ex(\RW_\e(t_1)^2)+\ldots+\Ex(\RW_\e(t_2-1)^2) }^{v/2}
\\
	&
	\leq C
	\sqBK{ \var_\e(t_1) + \ldots + \var_\e(t_2-1) }^{v/2}.
\end{align*}
Further using \eqref{eq:vartApprox} we thus obtain
$
	\Vert |f_2|^{v} \Vert_{2/v} \leq C (\e|t_2-t_1|)^{v/2},
$
thereby concluding \eqref{eq:hksumt}.

Proceeding to showing \eqref{eq:hkbdd}--\eqref{eq:hkx},
first we apply the inversion formula of the characteristic function,
$
	\hke(t_2,t_1,\zeta)
	=
	\frac{1}{2\pi\img} \int_{-\pi}^{\pi}
	e^{-\img\zeta r} \prod_{s=t_1}^{t_2-1}\phi(e^{\img r},\e;s) dr
$
and the uniform $ v $-H\"{o}lder continuity of $ x \mapsto e^{\img x} $, $ x\in\bbR $,
to obtain
\begin{align}
	&
	\label{eq:invshk}
	\hke(t_2,t_1,\xi)
	\leq
	C \int_{-\pi}^{\pi} \prod_{s=t_1}^{t_2-1} \absBK{ \phi(e^{\img r},\e;s) } dr,
\\
	&
	\label{eq:invshkx}
	\absBK{ \hke(t_2,t_1,\xi) - \hke(s_1,s_2,\xi') }
	\leq
	C \int_{-\pi}^{\pi} (|\xi-\xi'|r)^{v} \prod_{s=t_1}^{t_2-1} \absBK{ \phi(e^{\img r},\e;s) } dr.
\end{align}
To further bond these expressions,
we apply \eqref{eq:RWTaylor} for $ x = e^{\img r} $ to obtain
\begin{align}\label{eq:chasmall}
	\absBK{ \phi(e^{\img r},\e;s) }  \leq 1 - \e r^2/C,	
	\quad
	\forall s\in\N, \ \forall r\in\bbR \text{ with } |r|\leq r_0,
\end{align}
where $ r_0>0 $ is a constant.
As for $ |r|>r_0 $, we let $ f(n,\e;s) := \Pro(\RW_{\e}(s)=n-\drifte(s)) $,
whereby $ \phi(e^{\img r},\e;s) = \sum_{n\in\N} e^{\img r(n-\drifte(s))}f(n,\e;s) $.
Expressing the $ n=0 $ term as the sum of $ e^{-\img r\drifte(s)}f(1,\e;s)$
and $e^{-\img r\drifte(s)}(f(0,\e;s)-f(1,\e;s)) $,
and then combining the former with the $ n=1 $ term,
we obtain
\begin{align*}
	\phi(e^{\img r},\e;s)
	&=
	e^{-\img r\drifte(s)} (1+e^{\img r}) f(1,\e;s)
	+  e^{-\img r\drifte(s)}(f(0,\e;s)-f(1,\e;s))
\\
	&
	\quad
	+ \sum_{n>1} e^{\img r(n-\drifte(s))} f(n,\e;s).
\end{align*}
Taking the absolute value of this yields
\begin{align}\label{eq:cha}
	\absBK{ \phi(e^{\img r},\e;s) }
	\leq
	|1+e^{\img r}| f(1,\e;s)
	+  |f(0,\e;s)-f(1,\e;s)|
	+ \sum_{n>1} f(n,\e;s).
\end{align}
By \eqref{eq:RWt} and \eqref{eq:RW},
we find that $ f(0,\e;s) > f(1,\e;s) > \e/C $.
Using this and $ \sum_{n=0}^\infty f(n,\e;s) =1 $ in \eqref{eq:cha}, we then obtain the bound
\begin{align}\label{eq:chalarge}
	|\phi(e^{\img r},\e;s)|
	\leq
	f(1,\e;s) |1+e^{\img r}| + (1-2f(1,\e;s))
	\leq
	1- \e/C,
	\
	\forall s\in\N, \ |r|> r_0.
\end{align}
Now, combining \eqref{eq:chasmall} and \eqref{eq:chalarge},
we thus obtain $ |\phi(e^{\img r},\e;s)| \leq 1 - \e r^2/C $, for all $ r\in\bbR $ and $ s\in\N $.
Plugging this in \eqref{eq:invshk}--\eqref{eq:invshkx},
we conclude
$
	\prod_{s=t_1}^{t_2-1} \absBK{ \phi(e^{\img r},\e;s) }
	\leq
	 e^{-\e r^2(t_2-t_1)/C},
$
for all $ \e\in(0,1] $ and $ r\in(-\pi,\pi) $.
Using this, further integrating over $ r\in(-\pi,\pi) $,
we conclude \eqref{eq:hkbdd}--\eqref{eq:hkx}.
\end{proof}

We proceed to establish moment bounds on $ Z $.
Roughly speaking, we expect $ \mgg(t,\xi) $ (defined in~\eqref{eq:mgg}) to be effectively (as $ \e\to 0 $) delta-correlated in $ \xi $.
To see this, consider the pseudo generator
(as the true generator is $ \gen(s,\zeta):=\hke(s+1,s,\zeta) - \ind_\curBK{\zeta=0} $)
\begin{align}\label{eq:genps}
	\tag{4.22}
	\genps(s,\zeta) := \hke(s+1,s,\zeta) - \ind_\curBK{\zeta=-\drifte(s)},
	\quad
	\zeta \in (\N-\drifte(s)),
\end{align}
and rewrite $ \qv_1(s,\zeta) $ and $ \qv_2(s,\zeta) $ (as in \eqref{eq:qv1}--\eqref{eq:qv2}) as
\begin{align}
	&
	\tag{4.23}	
	\label{eq:qv1:gen}
	\qv_1(s,\zeta)
	= (\qe \cente(t) -1 ) Z(s,\zeta) - \big[ \genps(s) * Z(s) \big] (\zeta-\drift(s)),
\\
	&
	\tag{4.24}
	\label{eq:qv2:gen}
	\qv_2(s,\zeta)
	= (1-\cente(t) ) Z(s,\zeta) + \big[ \genps(s) * Z(s) \big] (\zeta-\drift(s)).	
\end{align}
By using \eqref{eq:RWt} and \eqref{eq:RW},
with $ (1-\qe) \leq C \e $,
it is not hard to show that
\begin{align}\label{eq:genBd}
	\tag{4.25}
	|\genps(s,\zeta)| \leq C \e \big( \tfrac{(\nu+\alphae(s))\dens}{1+\alphae(s)} \big)^{|\zeta|},
\end{align}
for some $ C<\infty $.
With $ \tfrac{(\nu+\alphae(s))\dens}{1+\alphae(s)} \leq \frac{\nu+\alpha}{1+\alpha} <1 $,
the r.h.s.\ of~\eqref{eq:genBd} decays exponentially in $ |\zeta| $, which suggests that 
$ \mgg(t,\xi) $, $ \xi\in\Xi(t) $, decorrelates for $ \xi $'s that are far apart. 
The preceding calculations are at the level of second moment.
As we will be working with arbitrarily high moments, we begin by preparing a technical result that exposes 
the aforementioned decorrelation structure.

\begin{remark}
\label{rmk:error}
The published version \cite{corwin17} of this article contains 
an erroneous application of Burkholder’s inequality in \cite[Lemma~4.3]{corwin17}.
We have remedied this issue in the present version
from Lemma~\ref{lem:mgg:decor} until the end of Section~\ref{sect:mom}.
\end{remark}

\begin{lemma}
\label{lem:mgg:decor}
Given $ s\in\N $, $ k\in\bbZ_{>0} $, and a deterministic function $ f: \Xi(s)\to\bbR $,
there exists $ C=C(k) $ such that
\begin{align}
\setcounter{section}{4}
\setcounter{equation}{25}
\label{eq:mgdecor}
	\BVert \sum_{\xi\in\Xi(s)} f(\xi) Z(s,\xi)\mgg(s,\xi) \BVert^2_{2k}
	\leq
	C \e^{2} \! \sum_{j=2,3} 
	\Big(
		\sum_{\vec{\xi}\in\Xi(s)^j}
		\prod_{i=1}^j e^{-\frac{|\xi_i-\xi_1|}{C}} |f(\xi_i)| \, \bVert Z(s,\xi_i) \bVert_{2k} 
	\Big)^{\frac{2}{j}}.
\end{align}
\end{lemma}
\begin{proof}
Fixing $ k\in\bbZ_{>0} $ and $ \vec{\xi}\in\Xi^{2k}_{\leq}(s) $,
throughout this proof we write $ C=C(k) $ to simplify notation.
Referring to the notation in~\eqref{eq:moveI},
we see that conditioning on $ \scrF(s) $ amounts to fixing $ \vec{g}=\vec{g}(s) $.
This being the case, we henceforward fix $ \vec{g}(s) $
and understand $ \Ex[\,\Cdot\,] $ as being taken with respect to the Bernoulli variables $ \B_i $ and $ \Bp_i $ in~\eqref{eq:moveI}.
Under the prescribed conventions, 
we set $ \moveICent_{n,m} := \moveI_{n,m} - \Ex[\moveI_{n,m}] $, fix $ \vec{n} = (n_1 \leq \ldots \leq n_{2k}) $, 
and write
\begin{align}
	\label{eq:mgdecor:1}
	\Ex\Big[\prod_{i=1}^k\moveCent_{n_i}(s) \Big| \scrF(s) \Big]
	=
	\sum_{\vec{m}\in\bbZ^{2k}(\vec{n})} \Ex\Big[ \prod_{i=1}^{2k} \moveICent_{n_i,m_i} \Big],
\end{align}
where $ \bbZ^{2k}(\vec{n}) := \{ (m_1,\ldots,m_{2k}) \in \bbZ^{2k} : m_i \leq n_i, \, i=1,\ldots,2k \} $.
Note that the r.h.s.\ of~\eqref{eq:mgdecor:1} is summable.
Indeed, from~\eqref{eq:moveI}, it is readily checked that
$	
	\Ex[|\moveI_{n,m}|^\ell]=\Ex[|\moveI_{n,m}|] 
 	\leq
	\exp(-\frac{1}{C}(n-m)),
$ 
$ \ell \in\bbZ_{>0} $.
In particular,
\begin{align}
	\label{eq:moveICentbd}
	\Ex\big[ | \moveICent_{n,m} |^{2k} \big] \leq C e^{-\frac{1}{C}|n-m|},
\end{align}
which ensures summability of the r.h.s.\ of~\eqref{eq:mgdecor:1}.

The first step of our proof is to bound the r.h.s.\ of~\eqref{eq:moveICentbd}.
For $ \vec{m}\in\bbZ^{2k}(\vec{n}) $, we say a pair $ m_i,m_{i'} $ of coordinates with $ i<i' $ are connected if $ m_{i'} \leq n_i $.
If $ m_i,m_{i'} $ are connected, all coordinates $ m_{i+1},\ldots,m_{i'-1} $ in between are also connected to $ m_i $ and $ m_{i'} $.
This gives rise to a partition of $ (m_1,\ldots,m_{2k}) $ into segments
\begin{align*}
	(m_{i})_{i\in g_j},	\quad g_{j} = [i_j,i_{j+1})\cap\bbZ,
	\quad
	1=i_1 < \ldots < i_{j(\vec{m})} = 2k,
\end{align*}
with the convention $ j_{j(\vec{m})+1} := 2k+1 $,
and each segment $ (m_i)_{i\in g_j} $ is formed by connected $ m_i $'s.
Referring to~\eqref{eq:moveI}, we have $ \moveICent_{n,m} \in \sigma(\B_i(s),\Bp_i(s),i=m,\ldots,n) $.
This gives
\begin{align}
	\label{eq:factorize}
	\Ex\Big[ \prod_{i=1}^{2k} \moveICent_{n_i,m_i} \Big] = \prod_{j=1}^{j(\vec{m})} \Ex\Big[ \prod_{i\in g_j}\moveICent_{n_i,m_i} \Big].
\end{align}
Further, since $ \Ex[\moveICent_{n,m}]=0 $, for the r.h.s.\ of~\eqref{eq:factorize} to be nonzero, we must have $ \# g_j >1 $.
On the r.h.s.\ of~\eqref{eq:factorize}, apply H\"{o}lder's inequality
$
	\Ex[ \prod_{i\in g_j} \moveICent_{n_i,m_i} ]
	\leq
	( \prod_{i\in g_j} \Ex[ |\moveICent_{n_i,m_i}|^{2k} ] )^{\frac{1}{2k}} 
$
and~\eqref{eq:moveICentbd} to get
\begin{align}
	\label{eq:factorize:bd}
	\Ex\Big[ \prod_{i=1}^{2k} \moveICent_{n_i,m_i} \Big] 
	\leq
	\prod_{j=1}^{j(\vec{m})} \prod_{i\in g_j}e^{-\frac{1}{C}(n_i-m_i)}.
\end{align}
Now sum the r.h.s.\ of~\eqref{eq:factorize:bd} over all $ \vec{m} \in \bbZ^{2k}(\vec{n}) $ with $ \# g_j >1 $. 
That $ (m_{i})_{i\in g_j} $ are connected implies $ m_{i}\leq n_{i_j} $, $ i\in g_j= [j_i,i_{j+1})\cap\bbZ $.
Hence, each sum over $ m_i $ with $ i\neq j_1,j_2,\ldots $ produces a factor of $ C \exp(-\frac{1}{C}(n_i-n_{i-1})) $;
while each sum over $ m_i $ with $ i=j_1,j_2,\ldots $ produces a factor of $ C $.
This gives
\begin{align}
	\label{eq:mgdecor:4}
	\Big| \Ex\Big[\prod_{i=1}^k\moveCent_{n_i}(s) \Big| \scrF(s) \Big] \Big|
	\leq
	C \sum_{\pi\in\Pair'_{2k}} \prod_{U\in\pi} e^{-\frac{1}{C}|n|_{U}},
\end{align}
Here $ \Pair'_{2k} $ denotes the set of partitions $ \pi $ of $ \{1,2,\ldots,2k\} $ into intervals $ U $ with $ \# U \geq 2 $,
and $ |n|_U := n_{i^*}-n_{i_*} $ for $ U=[i_*,i^*]\cap\bbZ $.
It will be convenient to consider only intervals $ U $ with $ \# U=2,3 $.
Indeed, a longer interval can be `chopped'  into smaller intervals,
for example for $ g=[1,5] $,
$
	e^{-\frac{1}{C}|n|_{g}}
	=
	e^{-\frac{1}{C}(n_{5}-n_1)}
	\leq
	e^{-\frac{1}{C}(n_{2}-n_1)} e^{-\frac{1}{C}(n_{5}-n_{3})}.
$
More generally, $ e^{-\frac{1}{C}|n|_{U}} \leq \prod_{i=1}^\ell e^{-\frac{1}{C}|n|_{U_i}} $,
for some $ U_i\subset U $ with $ \# U_i=2,3 $.
Using this in~\eqref{eq:mgdecor:4} gives
\begin{align}
	\label{eq:mgdecor:}
	\Big| \Ex\Big[\prod_{i=1}^k\moveCent_{n_i}(s) \Big| \scrF(s) \Big] \Big|
	\leq
	C \sum_{\pi\in\Pair_{2k}} \prod_{U\in\pi} e^{-\frac{1}{C}|n|_{U}},
\end{align}
where $ \Pair_{2k} $ denote the set of partitions $ \pi $ of $ \{1,2,\ldots,2k\} $ into intervals $ U $ with $ \#U = 2, 3 $.

Having established~\eqref{eq:mgdecor:}, we now proceed to show~\eqref{eq:mgdecor}.
Recall from~\eqref{eq:mgg} that $ \mgg(s,\xi) $ is given in term of $ \moveCent_{n}(s) $,
and that, under current scaling, $ |1-\qe| \leq C\e $ and $ |\cente(t)| \leq C $ (from~\eqref{eq:centte}).
Setting $ G:= \sum_{\xi\in\Xi(s)} f(\xi) Z(s,\xi)\mgg(s,\xi) $, we write
\begin{align}
	\label{eq:mgdecor:5}
	\Ex [ G^{2k} | \scrF(s) ]
	\leq
	C\e^{2k}
	\sum_{\vec{\xi}\in\Xi(s)^{2k}}
	\prod_{i=1}^{2k} |f(\xi_i) Z(s,\xi_i) | \cdot
	\Big| \Ex\Big[ \prod_{i=1}^{2k} \moveCent_{\xi_i+\Drift(s)}(s) \Big] \Big|.
\end{align}
Since the summand on the r.h.s.\ of~\eqref{eq:mgdecor:5} is symmetric in $ \xi_1,\ldots,\xi_{2k} $,
we may order the points $ \xi_1\leq\ldots\leq\xi_{2k} $ at a cost of a factor $ (2k)!  = C $.
Applying~\eqref{eq:mgdecor:} to the result, we get
\begin{align}
	\label{eq:mgdecor:6}
	\Ex [ G^{2k} | \scrF(s) ]
	\leq
	C\e^{2k}
	\sum_{\pi\in\Pair_{2k}} \prod_{U\in\pi} 
	\Big( \sum_{\xi_U\in \Xi^U_{\leq}(s)} e^{-\frac{1}{C}|\xi|_U} \prod_{i\in U} |f(\xi_i)| Z(s,\xi_i) \Big).
\end{align}
Here, for an interval $ U=[i,i']\cap\bbZ $, $ \Xi^U_{\leq}(s) := \{ \xi_{i} \leq \ldots \leq \xi_{i'}\in\Xi(s) \} $,
and $ |\xi|_U := |\xi_{i'}-\xi_i| $.
In~\eqref{eq:mgdecor:6}, apply Young's inequality $ \prod_{U\in\pi} |a_U| \leq \sum_{U\in\pi} \frac{\# U}{2k}|a_U|^{\frac{2k}{\# U}} $
and use $ \#g =2,3 $ to get
\begin{align*}
	\Ex [ G^{2k} | \scrF(s) ]
	&\leq
	C\e^{2k}
	\sum_{j=2,3}
	\Big( \sum_{\vec{\xi}\in \Xi^{[1,j]}_{\leq}(s)} e^{-\frac{1}{C}|\xi_j-\xi_1|} \prod_{i=1}^j|f(\xi_i)|Z(s,\xi_j) \Big)^{2k/j}.
\end{align*}
Further bound $ \exp(-\frac1C|\xi_j-\xi_1|) \leq \exp( -\frac{1}{jC} \sum_{i=1}^j |\xi_i-\xi_1| ) $,
and release the sum from $ \Xi^{[1,j]}_{\leq}(s) $ to $ \vec{\xi}\in \Xi(s)^j $.
Take $ \Ex[\,\Cdot\,] $ on both sides of the result, and take $ (\Cdot)^{1/k} $ on both sides using $ (\sum_{j=2,3}|a_j|)^{1/k} \leq 2 \sum_{j=2,3} |a_j|^{1/k} $.
This gives
\begin{align*}
	\Vert G \Vert^2_{2k}
	\leq
	&C \e^{2} \sum_{j=2,3} 
	\Big(	
		\BVert \sum_{\vec{\xi}\in\Xi(s)^j} \prod_{i=1}^j e^{-\frac{|\xi_i-\xi_1|}{C}} |f(\xi_i)|  Z(s,\xi_i) \BVert_{2k/j} \Big)^{2/j}.
\end{align*}
Using triangle inequality to pass $ \Vert\Cdot\Vert_{2k/j} $ into the sum (over $ \vec{\xi} $),
followed by applying H\"{o}lder's inequality:
$ \Vert \prod_{i=1}^j Z(s,\xi_i) \Vert_{2k/j} \leq \prod_{i=1}^j \Vert  Z(s,\xi_i) \Vert_{2k} $,
we conclude~\eqref{eq:mgdecor}.
\end{proof}

Recall the definition of~$ \Zmg(t_2,t_1,\xi) $ from~\eqref{eq:Zmg}.
Based on Lemma~\ref{lem:mgg:decor}, we proceed to establish bounds on $ \Zmg $.
Write $  \nabla^nf(\xi):= f(\xi+n)-f(\xi) $ for $ n $-step gradient, 
and set $ {\nabla^n_v}f(\xi) := (\e|n|)^{-v} \nabla^n f(\xi) \ind_{\{n\neq 0\}} $.
For a process $ G(s,\xi) $, $ s\in\N $, $ \xi\in\Xi(s) $, define the norm
\begin{align}
	\label{eq:norm}
	&\norm{G(t)}{2k,u,v} 
	:= 
	\sup_{\xi\in\Xi(t)} \sup_{n\in\bbZ} 
	\Big\{ e^{-u\e|\xi|} \Vert G(t,\xi) \Vert_{2k} + e^{-u\e|\xi|-u\e|\xi+n|}\,\Vert {\nabla^n_{v}}G(t,\xi) \Vert_{2k} \Big\}.
\end{align}
\begin{lemma}\label{lem:Zmg:}
For any $ T,u\in(0,\infty) $, $ k\geq 1 $ and $ v\in[0,\frac12] $, there exists $ C=C(u,k,T) $ such that
for all $ t_1\leq t_2\in\N\cap[0,\e^{-3}T] $,
\begin{align}
	\label{eq:Zmgbd1}
	\norm{ \Zmg(t_2,t_1) }{2k,u,v}^2
	\leq
	C\, \e^{3}\sum_{s=t_{1}}^{t_2-1} (\e^{3}(t_2-s+1))^{-\frac{1}2-v}  \norm{Z(s)}{2k,u,v}^2,
\end{align}
where $ \Zmg(t_2,t_1,\xi) $ is viewed as a process in $ (t_2,\xi) $.
\end{lemma}
\begin{proof}
Fix $ T,u\in(0,\infty) $, $ k\geq 1 $ and $ v\in[0,\frac12] $.
To simplify notation, throughout this proof we write $ C=C(u,k,T) $ and $ \xi_{+s} := \xi+\drift(s) $, and set
\begin{align}
	\label{eq:aux}
	\aux_a(s,\xi,G) 
	:= 
	(\e^{3}(t_2-s+1))^{-a} \sup_{n\in\bbZ} \big[ \hke(t_2,s+1)*(\Vert G(s) \Vert^2_{2k} + \Vert {\nabla^n_{v}}G(s) \Vert^2_{2k}) \big](\xi_{+s}).
\end{align}
We begin by deriving the following inequalities:
\begin{align}
	\label{eq:Zmg:model}
	\Vert \Zmg(t_2,t_1,\xi) \Vert^2_{2k}
	&\leq
	\e^{3}\sum_{s=t_1}^{t_2-1} \aux_{\frac{1}{2}+v}(s,\xi,Z),
\\
	\label{eq:Zxmg:model}	
	\Vert \nabla^n_{v}(\Zmg)(t_2,t_1,\xi) \Vert^2_{2k}
	&\leq
	\e^{3}\sum_{s=t_1}^{t_2-1} ( \aux_{\frac{1}{2}+v}(s,\xi,Z) + \aux_{\frac{1}{2}+v}(s,\xi+n,Z) \big),
\end{align}

To derive~\eqref{eq:Zmg:model},
let $ F_{\hke}(s,\zeta) := [ \hke(t_2,s+1)* \BK{ Z(s)\mgg(s) } ] (\zeta ) $,
and consider the discrete time martingales
$
 	M_{\hke}(t)
 	:= \textstyle \sum_{s=t_{1}}^{t-1} F(s,\xi_{+s}),
$
for $ t=t_1+1,\ldots,t_2 $.
Given that $ M_{\hke}(t_2) = \Zmg(t_2,t_1,\xi) $,
Burkholder's inequality applied to the martingale $ M_{\hke}(t) $ gives
\begin{align}
	\label{eq:Zmgbd:1}
	\Vertbk{ \Zmg(t_2,t_1,\xi) }_{2k}^2
	\leq
	C
	\BVert \sum_{s=t_{1}}^{t_2-1} F_{\hke}(s,\xi_{+s})^2 \BVert_k
	\leq
	C
	\sum_{s=t_{1}}^{t_2-1}
	\Vert F_{\hke}(s,\xi_{+s})  \Vert^2_{2k}.
\end{align}
On the r.h.s.\ of~\eqref{eq:Zmgbd:1},
applying Lemma~\ref{lem:mgg:decor} with $ f(s,\zeta)=\hke(t_2,s+1,\xi_{+s}-\zeta) $,
we have
\begin{align}
	\label{eq:Zmgfirst}
	\Vertbk{ \Zmg(t_2,t_1,\xi) }_{2k}^2
	\leq
	C\e^{2} \sum_{s=t_{1}}^{t_2-1}
	\sum_{j=2,3}
	\Big( 
	\sum_{\vec{\xi}\in\Xi(s)^j}
	\prod_{i=1}^j \hke(t_2,s+1,\xi_{+s}-\xi_i) \Vert Z(s,\xi_i) \Vert_{2k} \, g(\xi_i-\xi_1) \Big)^{\frac{2}{j}},
\end{align}
where $ g(\zeta) := \exp(-\frac{1}{C}|\zeta|) $.
On the r.h.s.\ of~\eqref{eq:Zmgfirst},
sum over $ \xi_3\in\Xi(s) $ (for $ j=3 $) using Cauchy--Schwarz inequality,
and bound $ \hke(t_2,s+1)^2 \leq \hk(t_2,s+1) C\, \e^{-\frac12}(t_2-s+1)^{-\frac12} $ (by~\eqref{eq:hkbdd}).
This gives
\begin{align}
	\notag
	&\sum_{\xi_3\in\Xi(s)} \Big( \hke(t_2,s+1,\xi_{+s}-\xi_3) \Vert Z(s,\xi_3) \Vert_{2k}\, g(\xi_3-\xi_1) \Big)
\\
	\label{eq:Zmgfirst:}
	&\leq
	\Big( \Big[ \frac{C\e^{-\frac12}\hke(t_2,s+1)}{\sqrt{t_2-s+1}}*\Vert Z(s) \Vert_{2k}^2 \Big](\xi_{+s}) \Big)^\frac12 
	\Big( \sum_{m\in\bbZ} g(m) \Big)^{\frac12}	
	\leq
	C\,\Big( \e \aux_{\frac12}(s,\xi,Z) \Big)^\frac12.
\end{align}
Next, for the sum over $ \xi_1,\xi_2\in\Xi(s) $ (for $ j=2,3 $),
write $ Z(s,\xi_2)\leq Z(s,\xi_1) + (\e n)^{v} |\nabla^n_vZ(s,\xi_1)| \leq (1+|n|)(Z(s,\xi_1)+|\nabla^n_vZ(s,\xi_1)|) $,
where $ n:= \xi_2-\xi_1 $,
and bound $ \hke(t_2,s+1,\xi-\xi_2) $ by $ C\e^{-\frac12}(t_2-s+1)^{-\frac12} $.
This gives
\begin{align*}
	&\sum_{\xi_1,\xi_2\in\Xi(s)} \prod_{i=1}^2\Big( \hke(t_2,s+1,\xi_{+s}-\xi_i) \Vert Z(s,\xi_i) \Vert_{2k}\, g(\xi_i-\xi_1) \Big)
\\
	&\leq
	\sup_{n\in\bbZ}
	\Big[ \hke(t_2,s+1)*\Big( \Vert Z(s)\Vert_{2k} \,\big(\Vert Z(s) \Vert_{2k}+\Vert \nabla^n_vZ(s) \Vert_{2k}\big)\Big)\Big](\xi_{+s}) \, 
	\sum_{m\in\bbZ} \frac{C\,\e^{-1/2}}{\sqrt{t_2-s+1}}(1+|m|)g(m).
\end{align*}
Indeed, the sum over $ m\in\bbZ $ contributes $ C\,\e^{-1/2}(t_2-s+1)^{-1/2} $.
Further bounding\\
$
	\Vert Z(s)\Vert_{2k} \,(\Vert Z(s) \Vert_{2k}+\Vert \nabla^n_vZ(s) \Vert_{2k})
	\leq
	C
	\Vert Z(s)\Vert_{2k}^2+ C\Vert \nabla^n_vZ(s) \Vert_{2k}^2,	
$
we have
\begin{align}
	\label{eq:Zmgfirst::}
	\sum_{\xi_1,\xi_2\in\Xi(s)} \prod_{i=1}^2\Big( \hke(t_2,s+1,\xi_{+s}-\xi_i) \Vert Z(s,\xi_i) \Vert_{2k}\, g(\xi_i-\xi_1) \Big)
	\leq
	C\,\e \aux_{\frac12}(s,\xi,Z).
\end{align}
Inserting~\eqref{eq:Zmgfirst:}--\eqref{eq:Zmgfirst::} into~\eqref{eq:Zmgfirst} yields
$
	\Vert \Zmg(t_2,t_1,\xi) \Vert^2_{2k}
	\leq
	\e^{3}\sum_{s=t_1}^{t_2-1} \aux_{\frac12}(s,\xi,Z).
$
Further, since $ t_2-s \leq \e^{-3}T $, we have 
$ \aux_{\frac12}(s,\xi,Z) \leq C\, \aux_{\frac{1}2+v}(s,\xi,Z) $,
and hence conclude~\eqref{eq:Zmg:model}.
The bound~\eqref{eq:Zxmg:model} follows the same argument as in the preceding, with $ (\hke)_{\nabla^n} $ in place of $ \hke $,
and with~\eqref{eq:hkx} in place of~\eqref{eq:hkbdd}.

Having derived~\eqref{eq:Zmg:model}--\eqref{eq:Zxmg:model},
we now proceed to show~\eqref{eq:Zmgbd1}.
Referring to the definition~\eqref{eq:norm} of $ \norm{\Cdot}{2k,u,v} $, we have
\begin{align*}
	\Vert Z(s,\zeta) \Vert^2_{2k} + \Vert \nabla^n_{v} Z(s,\zeta) \Vert^2_{2k} 
	\leq 
	C\, \norm{Z(s)}{2k,u,v}^2 \, e^{2u\e|\zeta|+2u\e|\zeta+n|}. 
\end{align*}
From the heat kernel bound~\eqref{eq:hksum} and with $ |\drift(s)| \leq C $, is it readily checked that
\begin{align}
	\label{eq:hksum:}
	\sum_{\zeta\in\Xi(s)} \hke(t_2,s+1,\xi_{+s}-\zeta) e^{\e u_1|\zeta|+\e u_2|\zeta+n|} \leq C(u_1,u_2) e^{\e u_1|\xi|+\e u_2|\xi+n|}.
\end{align}
Combining these bounds together gives 
\begin{align*}
	\aux_{\frac{1}2+v}(s,\xi,Z) \leq C\, (\e^{3}(t_2-s+1))^{-\frac12-v} \norm{Z(s)}{2k,u,v}^2 e^{2u\e|\xi|}.
\end{align*}
Inserting this bound into~\eqref{eq:Zmg:model}--\eqref{eq:Zxmg:model}, and summing the results together gives~\eqref{eq:Zmgbd1}.
\end{proof}

\begin{proof}[Proof of Proposition~\ref{prop:tight}]
Fix a collection of near equilibrium initial conditions,
with the corresponding $ u=u(k,v) $ (as in Definition~\ref{def:neareq}),
and fix $ T<\infty $, $ k\geq 1 $ and $ v\in[0,1/2) $.
We prove the following moment estimates:
\begin{align}
	&
	\tag{4.30}
	\label{eq:mom1}
	\Vertbk{ Z(t,r) }_{2k}
	\leq
	C e^{u\e|r|},
\\
	&
	\tag{4.31}
	\label{eq:momx1}
	\Vertbk{ Z(\tau,r) - Z(\tau,r') }_{2k}
	\leq
	C \BK{ \e|r-r'| }^{v} e^{u\e(|r|+|r'|)} ,
\\
	&
	\tag{4.32}
	\label{eq:momt1}
	\Vertbk{ Z(\tau,r) - Z(\tau',r) }_{2k}
	\leq
	C \BK{ \e^3|\tau'-\tau| }^{v/2} e^{2u\e|r|},
\end{align}
for some $ C=C(T,k,v)<\infty $
and for all $ \tau,\tau'\in [0,\e^{-3}T] $, $ r,r'\in\bbR $ and $ \e>0 $ small enough.
These estimates,
by the Kolmogorov--Chentsov criterion of tightness \cite[Corollary 14.9]{kallenberg02},
immediately imply the tightness of $ \{\Ze(\Cdot,\Cdot)\} $ in $ C(\bbR_+\times\bbR) $.

By definition, $ Z(\tau,r) $ is defined on $ \bbR_+\times\bbR $ by linear interpolation,
so without lost of generality we assume
$ \tau=t, \tau'=t'\in \N\cap[0,\e^{-3}T] $ and $ r=\xi, r'=\xi'\in\Xi(t) $,
and prove \eqref{eq:mom1}--\eqref{eq:momt1} as follows.
\end{proof}

\begin{proof}[Proof of \eqref{eq:mom1}--\eqref{eq:momx1}]
Recall that $ \norm{\Cdot}{2k,u,v} $ is a norm (see~\eqref{eq:norm}),
and in particular enjoys triangle inequality.
Take $ \norm{\Cdot}{2k,u,v} $ on both sides of~\eqref{eq:dSHE} using triangle inequality, square the result, and apply Lemma~\ref{lem:Zmg:}.
We get
\begin{align}
\label{eq:mom1:1}
\begin{split}
	\norm{Z(t_2)}{2k,u,v}^2
	&\leq
	\big( \norm{Z(t_1)}{2k,u,v} + \norm{Z(t_2,t_1)}{2k,u,v} \big)^2
\\
	&\leq
	C\,\norm{Z(t_1)}{2k,u,v}^2 
	+ 
	C\,\e^{3}\sum_{s=t_{1}}^{t_2-1} (\e^{3}(t_2-s+1))^{-\frac{1}2-v} \norm{Z(s)}{2k,u,v}^2.
\end{split}
\end{align}
Set $ w(t_2,t_1) := \max_{t\in[t_1,t_2]} \norm{Z(t_1)}{2k,u,v}^2 $.
Our goal is to show $ w(\e^{-3}T,0) \leq C $
(which implies the desired bounds~\eqref{eq:mom1}--\eqref{eq:momx1}).
To this end, take maximum of \eqref{eq:mom1:1} over $ t =t_1,t_1+1,\ldots,t_2 $ to get
\begin{align}
	\label{eq:mom1:w}
	w(t_2,t_1)
	\leq
	C\, \norm{Z(t_1)}{2k,u,v}^2
	+
	w(t_2,t_1) \, \max_{t\in[t_1,t_2]} V(t),
\end{align}
for some $ V(t) $ such that $ V(t)\leq C \e^{3}\sum_{s=t_{1}}^{t-1} (\e^{3}(t-s+1))^{-1/2-v} $.
This sum can be estimated via approximation with an integral, giving $ V(t) \leq C_* \, (\e^{3}(t_2-t_1+1))^{1/2-v} $,
for some constant $ C_*=C_*(u,k,T)<\infty $.
Given that $ v<1/2 $, we now fix $ T_*=T_*(v,u,k,T)>0 $ small enough so that $ C_*T_*^{1/2-v} < 1/2 $. 
This together with~\eqref{eq:mom1:w} gives
\begin{align}
	\label{eq:mom1:progresss}
	w(t_2,t_1)
	\leq
	C\, \norm{Z(t_1)}{2k,u,v}^2,
	\quad
	t_1\leq t_2 \in[0,\e^{-3}T],
	\
	t_2-t_1 \leq \e^{-3} T_*.
\end{align}
From~\eqref{eq:neareq:mom}--\eqref{eq:neareq:momx} and~\eqref{eq:hksum:},
it is readily checkd that $ \norm{Z(0)}{2k,u,v}^2 \leq C $.
Given this, starting from $ t=0 $,
we inductively progress by $ \e^{-3}T_* $ in time through~\eqref{eq:mom1:progresss}.
After $ \lceil T/T_* \rceil $ progressions, we arrive at $ w(\e^{-3}T,0) \leq C^{\lceil T/T_* \rceil }\norm{Z(0)}{2k,u,v}^2 = C $. 
\end{proof}

\begin{proof}[Proof of \eqref{eq:momt1}]
Without lost of generality we assume $ t'\leq t $.
By \eqref{eq:dSHE} we have
\begin{align*}
	Z(t,\xi) - Z(t',\xi)
	=
	\BK{ \Zdr(t,t',\xi) - Z(t',\xi) }
	+
	\Zmg(t,t',\xi).
\end{align*}
We bound separately
$ \Zdr^{*} := \Vertbk{ \Zdr(t,t',\xi) - Z(t,\xi) }_{2k} $ and
$ \Zmg^{*} := \Vertbk{ \Zmg(t,t',\xi)}_{2k} $.

For $ \Zdr^{*} $, with $ \sum_{\zeta\in\Xi(t')} \hke(t,t',\xi-\zeta) =1 $, we have
\begin{align*}
	\BK{ \Zdr(t,t',\xi) - Z(t',\xi) }
	=
	\sum_{\zeta\in\Xi(t')} \hke(t,t',\xi-\zeta)
	\BK{ Z(t',\zeta) - Z(t',\xi) }.
\end{align*}
Take $ \Vert\Cdot\Vert_{2k} $ on both sides,
and then use \eqref{eq:momx1} to bound $ \Vert Z(t',\zeta) - Z(t',\xi) \Vert_{2k} $
by
$
	C(\e|\xi-\zeta|)^v e^{u\e(|\xi|+|\zeta|)}
	\leq
	C(\e|\xi-\zeta|)^v e^{u\e|\xi-\zeta|+2u\e|\xi|}.
$
Using \eqref{eq:hksumt} to sum over $ \zeta $,
we then obtain the desired bound
$
	\Zdr^*
	\leq
	C (\e^3|t-t'|)^{v/2} e^{2u\e|\xi|}.
$

As for $ \Zmg^{*} $, combining \eqref{eq:Zmgbd1} for $ v\mapsto \frac12-v $ and $ \norm{Z(s)}{2k,u,v} \leq C $ 
(as shown previously) gives
$
	(\Zmg^*)^2
	\leq
	C e^{2u\e|\xi|}	
	\e^{3} \sum_{s=t'}^{t-1} (\e^{3}(t-s+1))^{-1+v}.
$
The last sum can be estimated by approximation with an integral, giving the desired bound
$
	(\Zmg^{*})^2
	\leq
	C e^{2u\e|\xi|} (\e^3|t-t'|)^{v}.
$
\end{proof}
%
%
%
This completes the proof of Proposition~\ref{prop:tight}.
We now turn to the proof of Proposition~\ref{prop:stepEst}.

\begin{proof}[Proof of Proposition~\ref{prop:stepEst}]
Fix $ k\in\N $, $ v\in(0,\frac12) $, and $ T\in(0,\infty) $.
As explained at the beginning of the proof of Proposition~\ref{prop:tight},
without lost of generality we let $ \tau=t_2\in\N\cap (0,\e^{-3}T] $
and $ r=\xi, r'=\xi'\in\Xi(t) $.
To simplify notation we write $ C=C(k,v,T) $.

Recall that $ \Zd(t,\xi) := \xcent \e^{-1}(1-\dens) Z(t,\xi) $.
Similarly define $ \Zddr(t,\xi) := [\hk(t,0)*\Zd(0)](\xi) = \xcent \e^{-1}(1-\dens) \Zdr(t,\xi) $ 
and $ \Zdmg(t,\xi) := \xcent \e^{-1}(1-\dens) \Zmg(t,0,\xi) $ so that $ \Zd=\Zddr+\Zdmg $.
The first step of the proof is to derive a bound similarly to the one in Lemma~\ref{lem:Zmg:}.
To this end, for a process $ G(t,\xi) $, $ t\in\N $, $ \xi\in\Xi(t) $, consider
\begin{align*}
	\normd{G(t)}{k,v}
	:=
	(\e^{3}(t+1))^{1+v} \norm{G(t)}{2k,0,v}^2
	+
	(\e^{3}(t+1))^{\frac12+v} \
	\e \!\! \sum_{\xi\in\Xi(t)} \Big( \Vert G(t,\xi) \Vert^2_{2k} + \sup_{n\in\bbZ} \Vert {\nabla^n_{v}}G(t,\xi,n) \Vert^2_{2k} \Big).
\end{align*}
Recall the definition of $ \aux_a(s,\xi,\Cdot) $ from~\eqref{eq:aux}.
In~\eqref{eq:Zmg:model}--\eqref{eq:Zxmg:model}, set $ t_1=0 $,
and multiply both sides of by $ \xcent^2 \e^{-2}(1-\dens)^2 $
so that each $ Z $, $ \Zmg $ is replaced by $ \Zd $, $ \Zdmg $ therein.
Adding the results together, here we have
\begin{align*}
	\Vert \Zdmg(t_2,\xi) \Vert^2_{2k}
	+
	\Vert \nabla^n_{v}(\Zdmg)(t_2,\xi) \Vert^2_{2k}
	\leq
	C\, \e^{3}\sum_{s=0}^{t_2-1} ( \aux_{\frac{1}{2}+v}(s,\xi,\Zd) + \aux_{\frac{1}{2}+v}(s,\xi+n,\Zd)).
\end{align*}
From this and the definition of $ \normd{\Cdot}{k,v} $, we have
\begin{align}
	\label{eq:Zdmg:model}
	\normd{Z(t_2)}{k,v}
	\leq
	C \, \e^{3}\sum_{s=t_1}^{t_2-1} \big( (\e^3(t_2+1))^{1+v} \, \aux^\infty(s) + (\e^3(t_2+1))^{\frac12+v} \,\aux^\Sigma(s) \big),
\end{align}
where $ \aux^\infty(s) := \sup_{\xi\in\Xi(t_2)}\aux_{\frac{1}{2}+v}(s,\xi,\Zd) $
and $ \aux^\Sigma(s) := \sum_{\xi\in\Xi(t_2)}\aux_{\frac{1}{2}+v}(s,\xi,\Zd) $.
To bound the r.h.s.\ of~\eqref{eq:Zdmg:model}, we write
\begin{align}
	\label{eq:auxbd:sum}
	&\sum_{\zeta\in\Xi(s,\zeta)} (\Vert \Zd(s) \Vert^2_{2k}+\Vert \nabla^n_v \Zd(s) \Vert^2_{2k}) \leq \normd{\Zd(s)}{k,v} \, (\e^{3}(s+1))^{-\frac12-v},
\\
	\label{eq:auxbd:sup}
	&\sup_{\zeta\in\Xi(s,\zeta)} (\Vert \Zd(s) \Vert^2_{2k}+\Vert \nabla^n_v \Zd(s) \Vert^2_{2k}) \leq \normd{\Zd(s)}{k,v} \, (\e^{3}(s+1))^{-1-v}.
\end{align}
In~\eqref{eq:aux}, set $ G=\Zd $ and $ a=\frac12+v $, use~\eqref{eq:hkbdd} to bound $ \hke(t_2,s+1) \leq C \e^{-\frac12}(t_2-s+1)^{-\frac12} $,
and use~\eqref{eq:auxbd:sum} to bound the remaining sum. This gives
\begin{align*}
	\aux^\infty(s) \leq  C\,(\e^3(t-s+1))^{-1-v}(\e^3(s+1))^{-\frac12-v} \normd{Z(s)}{k,v}.
\end{align*}
Similarly, using~\eqref{eq:auxbd:sup} and $ \sum_{\xi} \hke(t_2,s+1,\xi)=1 $ yields
\begin{align*}
	\aux^\infty(s) \leq C\,(\e^3(t-s+1))^{-\frac12-v}(\e^3(s+1))^{-1-v} \normd{Z(s)}{k,v}.
\end{align*}
In~\eqref{eq:aux}, summing over $ \xi\in \Xi(t_2) $,
and using~\eqref{eq:auxbd:sum} and $ \sum_{\xi} \hke(t_2,s+1,\xi)=1 $, we have
\begin{align*}
	\aux^\Sigma(s) \leq C\, (\e^3(t-s+1))^{-\frac12-v}(\e^3(s+1))^{-\frac12-v} \normd{Z(s)}{k,v}.
\end{align*}
Combining the preceding bounds on $ \aux^\infty(s) $ and $ \aux^\Sigma(s) $ gives
\begin{align}
	\label{eq:Zdmg:model:}
	\normd{Z(t_2)}{k,v}
	\leq
	C \, \e^{3}\sum_{s=t_1}^{t_2-1} V_*(s) \normd{Z(s)}{k,v},
\end{align}
where, with the notation $ V_{a,b}(s) := (\e^3(t_2-s+1))^{-a}(\e^3(s+1))^{-b} $,
\begin{align*}
	V_*(s) := (\e^3(t_2+1))^{1+v} \big( V_{1+v,\frac12+v}(s) \wedge V_{\frac12+v,1+v}(s) \big) + (\e^3(t_2+1))^{\frac12+v} V_{\frac12+v,\frac12+v}(s).
\end{align*}

Referring to the definition of $ \normd{\Cdot}{k,v} $, we have that 
$ \normd{G_1+G_2}{k,v} \leq 2\normd{G_1}{k,v}+2\normd{G_2}{k,v} $.
Using this and~\eqref{eq:Zdmg:model:} gives 
\begin{align}
	\label{eq:Zdbd}
	\normd{\Zd(t_2)}{k,v} \leq 2 \normd{\Zddr(t_2)}{k,v} + 2 \normd{\Zdmg(t_2)}{k,v}
	\leq
	2 \normd{\Zddr(t_2)}{k,v} + C\, \e^{3}\sum_{s=0}^{t_2-1} V_*(s) \normd{Z(s)}{k,v}.
\end{align}
From~\eqref{eq:hkbdd}--\eqref{eq:hkx}, $ \sum_{\xi} \hk(t_2,0,\xi)=1 $, and \eqref{eq:Zdmass}, it is readily checked that
\begin{align*}
	\Big|\e\sum_{\xi\in \Xi(t_2)} \Zddr(t_2,\xi)\Big| \leq C,
	\quad
	|\Zddr(t_2,\xi)| \leq \frac{C}{(\e^3(t_2+1))^{1/2}},
	\quad
	|\nabla^n \Zddr(t_2,\xi)| \leq \frac{ C\, (\e|n|)^{2v} }{ (\e^3(t_2+1))^{1/2+v} }.
\end{align*}
Using these bounds together with $ |\nabla^n\Zdr(t_2,\xi)|^2 \leq (|\Zddr(t_2,\xi)|+|\Zddr(t_2,\xi+n)|)|\nabla^n \Zddr(t,\xi)| $ gives
\begin{align}
	\label{eq:Zddrt2}
	\normd{\Zddr(t_2)}{k,v} \leq  C.
\end{align}
Next, the sum $ \e^{3} \sum_{s=0}^{t-1} V_*(s) $ can be estimated by approximating with an integral,
yielding
\begin{align}
	\label{eq:V*bd}
	\e^{3}\sum_{s=0}^{t-1} V_*(s) 
	\leq 
	C\, \big(  (\e^{3}(t+1))^{\frac12-v} + (\e^{3}(t+1))^{1-v} \big)
	\leq 
	C\, (\e^{3}(t+1))^{\frac12-v}.	
\end{align}
Now, set $ w(t_2) := \max_{t\in[0,t_2]} \normd{\Zd(t)}{k,v} $.
Take maximum of~\eqref{eq:Zdbd} over $ t=0,1,\ldots, t_2 $,
and use the bounds~\eqref{eq:Zddrt2}--\eqref{eq:V*bd} in~\eqref{eq:Zdbd}.
We have
\begin{align*}
	w(t_2) \leq C + w(t_2)\, C (\e^{3}(t_2+1))^{\frac12-v}.
\end{align*}
With $ v<\frac12 $, we fix $ T_*>0 $  small enough so that $ C T_*^{\frac12-v} < \frac12 $.
This gives $ w(t) \leq C $, $ t \leq \e^{-3} T_* $,
and hence the desired bounds~\eqref{eq:momd}--\eqref{eq:momdx} for $ t\in [0,\e^{-3}T_*] $.
Generalization to $ t\in [\e^{-3}T_*,\e^{-3}T] $ is immediate.
Instead of $ t=0 $, one initiates the process at time $ t_*= \lfloor \e^{-3}T_* \rfloor $.
The bounds~\eqref{eq:momd}--\eqref{eq:momdx} at $ t=t_* $
ensures that $ Z(t_*,\xi) $ is near equilibrium (i.e., satisfying~\eqref{eq:neareq:mom}--\eqref{eq:neareq:momx}).
This being the case, the moment bounds~\eqref{eq:mom1}--\eqref{eq:momx1} applies for $ t\in[t_*,\e^{-3}T] $.
\end{proof}

\section{The Martingale Problem: Proof of Proposition~\ref{prop:unique}}
\label{sect:mg}

Hereafter we use $ \bd(s,\veczeta) $ and $ \er(s,\veczeta) $, $ \veczeta=(\zeta_1,\ldots,\zeta_n) $,
to denote respectively \emph{generic} processes
that are uniformly bounded and uniformly vanishing, i.e.\
\begin{align*}
	&
	\sup\curBK{
		\Vertbk{\bd(s,\veczeta)}_k:
		\veczeta\in(\e^{-1}U)^n, s\leq T\e^{-3}, \e\in(0,1]
		}
	 <\infty,
\\
	&
	\sup\curBK{
		\Vertbk{\er(s,\veczeta)}_k :
		\veczeta\in(\e^{-1}U)^n, s\leq T\e^{-3}
		}
	\longrightarrow 0,
	\quad
	\text{ as } \e\to 0,
\end{align*}
for any compact $ U\subset \bbR $, $ k\geq 1 $ and $ T>0 $.

We begin by deriving an approximate expression for
the cross variance as in \eqref{eq:qv}.

\begin{lemma}\label{lem:qvApprox}
For near equilibrium initial conditions, we have
\begin{align}\label{eq:qvApprox}
\begin{split}
	&
	Z(s,\zeta_1) Z(s,\zeta_2) \Ex\BK{ \mgg(s,\zeta_1) \mgg(s,\zeta_2) \middle| \scrF(s) }
\\
	&
	\quad
	=
	\e^{2} \frac{\alpha\const}{(1+\alpha\const)^2}
	\BK{ \frac{(\nu+\alpha(s))\dens}{1+\alpha(s)} }^{|\zeta_1-\zeta_2|}
	\BK{ Z(s,\zeta_1\wedge\zeta_2)^2 + \er(s,\zeta_1,\zeta_2)}.
\end{split}
\end{align}
\end{lemma}

\begin{proof}
We prove \eqref{eq:qvApprox} by approximating the identities \eqref{eq:qv1:gen}--\eqref{eq:qv2:gen},
using the moment estimate \eqref{eq:momx1}.
By \eqref{eq:centte} we have $ (\qe\cente(t)-1) = -\e(1+\alpha\const)^{-1} + O(\e^2) $,
so that $ (\qe\cente(t)+1)Z(s,\zeta) = -\e(1+\alpha\const)^{-1}Z(s,\zeta) + \e\er(s,\zeta) $,
and by \eqref{eq:momx1}, fixing arbitrary $ v\in(0,1/2) $,
we have 
$ 
	Z(s,\zeta') = Z(s,\zeta) + |\e(\zeta'-\zeta)|^{v} \bd(s,\zeta,\zeta') 
$.
In \eqref{eq:qv1:gen}, using these approximations we obtain
\begin{align}
	&
	\notag
	\qv_1(s,\zeta)
	=
	-\e(1+\alpha\const)^{-1} Z(s,\zeta) + \e \er(t,\zeta)
\\
	&
	\label{eq:qv1Approx:}
	\quad
	- \Big( \sum_{\zeta'\in\Xi(s)} \genps(s,\zeta_{-s}-\zeta') \Big) Z(s,\zeta)
	+ \sum_{\zeta\in\Xi(s)} \genps(s,\zeta_{-s}-\zeta') |\e(\zeta'-\zeta)|^{v} \bd(s,\zeta,\zeta') .
\end{align}
With $ \genps(s,\zeta_{-s}-\zeta') $ as in \eqref{eq:genps},
the second last term in \eqref{eq:qv1Approx:} is zero since $ \gen(s,\zeta) $,
and the last term is of the form $ \e^{1+v} \bd(s,\zeta) \leq \e \er(s,\zeta) $ by \eqref{eq:genBd}.
Consequently,
\begin{align}\label{eq:qv1Approx}
	\qv_1(s,\zeta)
	=
	-\e(1+\alpha\const)^{-1} Z(s,\zeta) + \e \er(t,\zeta).
\end{align}
Similarly, for $ \qv_2(s,\xi_2) $ we have
\begin{align}\label{eq:qv2Approx}
	\qv_2(s,\zeta)
	= -\e \alpha\const(1+\alpha\const)^{-1} Z(s,\zeta) + \e \er(s,\zeta).
\end{align}
Combining \eqref{eq:qv1Approx}--\eqref{eq:qv2Approx} with \eqref{eq:qv}
yields \eqref{eq:qvApprox}.
\end{proof}

We proceed to proving Proposition~\ref{prop:unique}.
Recall from \cite{bertini97} the following martingale problem of the \ac{SHE}.

\begin{definition}\label{defn:MGprob}
Let $\Zlim$ be a $C([0,\infty),C(\bbR))$-valued process such that
given any $T>0$, there exists $u<\infty$ such that
\begin{equation}\label{eq:mgprob:mom}
	\sup_{\tau\in[0,T]}
	\sup_{r\in \bbR} e^{-u|r|} \Ex \BK{\Zlim(\tau,r)^2}
	<
	\infty.
\end{equation}
For such $ \Zlim $ and for $ \psi\in C^\infty_c(\bbR) $,
let $ \anglebk{\Zlim(\tau),\psi} := \int_{\bbR} \Zlim(\tau,r) \psi(r) dr $.
We say $ \Zlim $ solves the martingale problem with initial condition
$ \Zic\in C(\bbR) $ if $ \Zlim(0,\Cdot)=\Zic(\Cdot) $ in distribution, and
\begin{align*}
	&
	\tau
	\longmapsto
	N_\psi(\tau)
	:=
	\angleBK{ \Zlim(\tau),\psi }
	- \angleBK{ \Zlim(0),\psi }
	- \int_0^\tau 2^{-1} \angleBK{ \Zlim(\tau'),\tfrac{d^2}{dx^2} \psi } d\tau'.
\\
	&
	\tau
	\longmapsto	
	\hatN_{\psi}(\tau)
	:=
	 (N_\psi(\tau))^2
	 - \int_0^\tau \angleBK{ \Zlim(\tau')^2, \psi^2 } d\tau'
\end{align*}
are local martingales, for any $\psi\in C^\infty_c(\bbR)$.
\end{definition}

\begin{proof}[Proof of Proposition~\ref{prop:unique}]
Recall from \cite[Proposition~4.11]{bertini97} that
for any initial condition $ \Zic $ satisfying
\begin{align}\label{eq:mgic}
	\VertBK{ \Zic(r) }_{2} \leq C e^{a|r|},	\quad\text{for some }a >0,
\end{align}
the martingale problem of Definition~\ref{defn:MGprob} has a unique solution,
which coincides with the law of the solution of the \ac{SHE}
with initial condition $ \Zic $.
By passing to the relative subsequence,
we assume that
$ \Ze \Rightarrow \Zlim $,
which, by \eqref{eq:mom1}, satisfies the moment condition \eqref{eq:mgprob:mom}.
It hence suffices to show that $ \Zlim $
solves the martingale problem in Definition~\ref{defn:MGprob}.

We begin by deriving the discrete analog of $ N_\psi(\tau) $ and $ \hatN_\psi(\tau) $.
To this end, fixing $ \psi\in C^\infty_c(\bbR) $,
we consider the discrete approximation
\begin{align}\label{eq:Zpsie}
	\angleBK{ Z(t), \psi }_\e := \e \xcent^{-1} \sum_{\xi\in\Xi(t)} Z(t,\xi) \psi( \e\xcent^{-1}\xi )
\end{align}
of $ \anglebk{ \Ze( \e^{3}(\tcente J)^{-1}t ), \psi } $,
and similarly define
\begin{align}\label{eq:mgpsi}
	M_\psi(t) :=
	\e \xcent^{-1} \sum_{\xi\in\Xi(t+1)} Z(t,\xi+\drift(t))
	\mgg(t,\xi+\drift(t)) \psi( \e\xcent^{-1}\xi ).
\end{align}
In \eqref{eq:dSHE:}, multiply both sides by $ \e\xcent^{-1}\psi( \e\xcent^{-1}\xi ) $.
Upon summing over $ \xi \in \Xi(t+1) $,
we obtain
$
	\anglebk{ Z(t+1), \psi }_\e =
	\anglebk{ Z(t), \psi_{\hke(t+1,t)} }_\e + M_\psi(t),
$
where
\begin{align*}
	\psi_{\hke(t+1,t)}(\zeta) := 
	\sum\nolimits_{\xi\in\Xi(t+1)} \hke(t+1,t,\xi-\zeta) \psi(\e^{-1}\xcent\xi).
\end{align*}
Subtracting $ \anglebk{Z(t),\psi}_\e $ from both sides,
we further obtain
\begin{align}\label{eq:dmgprob}
	\angleBK{ Z(t+1), \psi }_\e - \angleBK{ Z(t), \psi }_\e
	=
	\anglebk{ Z(t), \psi_{\gen(t)} }_\e + M_\psi(t,\xi),
\end{align}
where
\begin{align}\label{eq:psigen}
	\psi_{\gen(t)}(\zeta) 
	:= 
	\sum_{\xi\in\Xi(t+1)} \hke(t+1,t,\xi-\zeta) \psi(\e^{-1}\xcent\xi) - \psi(\e^{-1}\xcent\zeta).
\end{align}

Now, summing \eqref{eq:dmgprob} over $ s=0,\ldots,t-1 $,
we arrive at
\begin{align}\label{eq:mgprob:d1}
	\angleBK{ Z(t), \psi }_\e - \angleBK{ Z(0), \psi }_\e
	-
	\sum_{s=0}^{t-1} \angleBK{ Z(s), \psi_{\gen(t)} }_\e
	=
	\sum_{s=0}^{t-1} \mg_\psi(s)
	:= \Ne_\psi(t).
\end{align}
The process $ t\mapsto \Ne_\psi(t) $ is a discrete time martingale
of quadratic variation\\
$ \sum_{s=0}^{t-1} \Ex( \mg_\psi(s)^2 | \scrF(s) ) $,
so
\begin{align}\label{eq:mgprob:d2}
	\hatNe_\psi(t)
	:=
	\BK{ \Ne_\psi(t) }^2
	- \sum_{s=0}^{t-1} \Ex\BK{  \mg_\psi(s)^2 \middle| \scrF(s) }
\end{align}
is also a discrete time martingale.

With $ \Ne_\psi(t) $ and $ \hatNe_\psi(t) $ as in the preceding,
we proceed to showing
that $ N_\psi(\tau) $ and $ \hatN_\psi(\tau) $ are local martingales.
Since, by \eqref{eq:momt1},
passing from discrete time to continuous time
introduces only a negligible error,
it suffices to show that terms in \eqref{eq:mgprob:d1}--\eqref{eq:mgprob:d2}
converge in distribution to the corresponding terms.
More precisely, recalling $ \te(\tau) :=\e^{-3}\tcente J \tau $, our goal is to show
\begin{align*}
	\angleBK{Z(\te(\tau)),\psi}_\e
	&\Longrightarrow
	\angleBK{ \Zlim(\tau), \psi },
\\
	\sum_{s<\te(\tau)} \angleBK{ Z(s), \psi_{\gen} }_\e
	&\Longrightarrow
	\int_0^\tau 2^{-1} \angleBK{ \Zlim(\tau'), \frac{d^2\psi}{dx^2} } d\tau',
\\
	\sum_{s<\te(\tau)} \Ex\BK{  \mg_\psi(s)^2 \middle| \scrF(s) }
	&\Longrightarrow  	
	\int_0^\tau \angleBK{ \Zlim^2(\tau'), \psi(\tau')^2 } d\tau'.
\end{align*}
To this end,
since $ \Ze \Rightarrow \Zlim $, it clearly suffices to show
\begin{align}
	\label{eq:cnvg:Zt}
	&
	\Ex \Big|
		\angleBK{Z(\te(\tau)),\psi}_\e
		- \angleBK{ \Ze(\tau), \psi }
	\Big|
	\longrightarrow 0,
\\
	\label{eq:cnvg:drift}
	&
	\Ex \Big|
		\sum_{s<\te(\tau)} \angleBK{ Z(s), \psi_{\gen} }_\e
		-
		\int_0^\tau \angleBK{ \Ze(\tau'), \frac{d^2\psi}{dx^2} } d\tau'
	\Big|
	\longrightarrow 0,
\\
	\label{eq:cnvg:mg}
	&
	\Ex \Big|
		\sum_{s<\te(\tau)} \Ex\BK{  \mg_\psi(s)^2 \middle| \scrF(s) }
		-
		\int_0^\tau \angleBK{ \Ze^2(\tau'), \psi^2 } d\tau'
		\Big|
	\longrightarrow
	0.
\end{align}
We prove \eqref{eq:cnvg:Zt}--\eqref{eq:cnvg:mg} as follows.
\end{proof}

\begin{proof}[Proof of \eqref{eq:cnvg:Zt}]
This amounts to show that the terms
\begin{align}\label{eq:Zpsi}
	\angleBK{ \Ze(\tau), \psi }
	=
	\e \xcent^{-1} \int_{\bbR} Z(\te(\tau),r) \psi(\e \xcent^{-1}r) dr
\end{align}
and $ \angleBK{ Z(\te(\tau)), \psi }_\e $ are approximately equal.
To this end,
fixing arbitrary $ \zeta\in\Xi(t) $ and $ |r-\zeta|\leq 1 $,
we use the smoothness of $ \psi $
and the moment estimates \eqref{eq:mom1}--\eqref{eq:momx1}, for arbitrary $ v\in(0,1/2) $,
to obtain
$
	 Z(t,\zeta) \psi(\e\xcent^{-1}\zeta)
	- Z(t,r) \psi(\e\xcent^{-1} r)
	= \e^v\bd(\zeta,r)
	= \er(t,\zeta,r).
$
From this, with $ \anglebk{ \Ze(\tau), \psi } $ and $ \anglebk{ Z(\te(\tau)), \psi }_\e $
as in \eqref{eq:Zpsi} and \eqref{eq:Zpsie}, we conclude
$
	\anglebk{ \Ze(\tau), \psi }
	=
	\anglebk{ Z(\te(\tau)), \psi }_\e + \er(t),
$
whereby \eqref{eq:cnvg:Zt} follows.
\end{proof}

\begin{proof}[Proof of \eqref{eq:cnvg:drift}]
Taylor expanding $ \psi(\xi\e\xcent^{-1}) $ around $ \xi=\zeta $ yields
\begin{align*}
	\psi(\e\xcent^{-1}\xi)
	=&
	\psi(\e\xcent^{-1}\zeta)
	+ \BK{ \frac{d\psi}{dx}(\e\xcent^{-1}\zeta)} \e\xcent^{-1}(\zeta-\xi)
\\
	&
	+ \BK{ 2^{-1} \BK{ \frac{d^2\psi}{dx^2}(\e\xcent^{-1}\zeta)} + \er(\xi,\zeta) }
	\e^2\xcent^{-2} (\zeta-\xi)^2.
\end{align*}
Plug this in \eqref{eq:psigen}. With
$
	\sum_{\xi\in\Xi(s+1)} \hke(s+1,s,\xi-\zeta)(\xi-\zeta)^k = \Ex(\RW_\e(s)^k),
$
using $ \Ex(\RW(s))=0 $, \eqref{eq:RWvar} and $ \vare \leq C\e $,
we obtain
\begin{align*}
	\psi_{\gen(s)} (\zeta)
	=
	2^{-1} \BK{ \frac{d^2\psi}{dx^2}(\e\xcent^{-1}\zeta)} \e^{2} \vare(s)
	+
	\e^3 \er(s,\zeta). 	
\end{align*}
Now, plugging this expression of $ \psi_{\gen(s)} (\zeta) $ in the l.h.s.\ of \eqref{eq:cnvg:drift},
with $ \te(\tau) \leq \e^{-3} C $,
we obtain
\begin{align*}
	\sum_{s<\te(\tau)} \angleBK{ Z(s), \psi_{\gen(s)} }_\e
	=
	\sum_{s<\te(\tau)} \e^{2} \vare(s) \angleBK{ Z(s), 2^{-1}\frac{d^2\psi}{dx^2} }_\e + \er.
\end{align*}
Next, divide the sum on the r.h.s.\
into sums over the disjoint intervals $ T_t := \bbZ\cap[tJ,tJ+J) $.
We neglect the boundary terms of $ T_{\tau\e^{-3}/\tcente} \cap [0,\te(\tau)) $,
since, by \eqref{eq:mom1}, those terms contribute only $ \e^{2}\vare(s) \bd = \er $.
Within each interval $ T_t $, use \eqref{eq:momt1} to replace
$ \anglebk{ Z(s), \frac{d^2\psi}{dx^2} }_\e $
with $ \anglebk{ Z(t J), \frac{d^2\psi}{dx^2} }_\e + \er(s) $.
Further, with $ \vare(s) $ and $ \tcente $ as in \eqref{eq:vart} and \eqref{eq:tcent},
we have $ \sum_{s\in T_t} \vare(s) = \e(\tcente)^{-1} $.
Consequently,
\begin{align}\label{eq:cnvgdr:}
	\sum_{s<\te(\tau)} \angleBK{ Z(s), \gen\psi }_\e
	=
	\frac{\e^{3}}{\tcente} \sum_{t < \e^{-3}\tcente \tau} \angleBK{ Z(tJ), 2^{-1}\frac{d^2\psi}{dx^2} }_\e + \er.
\end{align}

The r.h.s.\ of \eqref{eq:cnvgdr:} represents a discrete approximation
of $ \int_0^\tau \anglebk{ \Ze(\tau'),2^{-1}\frac{d^2\psi}{dx^2} } d \tau' $.
In particular, by following the same procedure as in the proof of \eqref{eq:cnvg:Zt},
one obtains\\
$
	\frac{\e^{3}}{\tcente} \sum_{t < \te(\tau)}
	\anglebk{ Z(tJ), 2^{-1}\frac{d^2\psi}{dx^2} }_\e
	=
	\int_0^{\tau}
	\anglebk{Z(\tau'),2^{-1}\frac{d^2\psi}{dx^2}} d \tau'
	+
	\er,
$
thereby concluding \eqref{eq:cnvg:drift}.
\end{proof}

\begin{proof}[Proof of \eqref{eq:cnvg:mg}]
To calculate $ 	\Ex( M_\psi(s)^2 | \scrF(s) ) $,
we use the expression \eqref{eq:mgpsi}
and the approximation \eqref{eq:qvApprox} to obtain
\begin{align}\label{eq:cnvgmg:id}
	&
	\Ex\BK{ M_\psi(s)^2 \middle| \scrF(s) }
	=
	\frac{\e^4\alpha\const}{(1+\alpha\const)^2\xcent^2}
	\sum_{\xi\in\Xi(s+1)} Z(s,\xi_{+s})^2 \psi(\e\xcent^{-1} \xi)
	F(s,\xi)
	+
	\e^{3} \er(s),
\end{align}
where
\begin{align}\label{eq:cnvgmg:F}
	F(s,\xi) :=
	\sum_{n\in\bbZ}
	\psi(\e\xcent^{-1}(\xi+|n|))
	\BK{ \frac{(\nu+\alphae(s))\dens}{1+\alphae(s)} }^{|n|}.
\end{align}
Let $ \eta_\e := \sum_{n\in\bbZ} ( \frac{(\nu+\alphae(s))\dens}{1+\alphae(s)} )^{|n|}  $.
In \eqref{eq:cnvgmg:F},
using the continuity of $ \psi $ at $ \e\xcent^{-1}\xi $,
we further obtain
$ 	
	F(s,\xi) = \eta_\e \psi(\e\xcent^{-1}\xi) + \er(s,\xi).
$
Plugging this expression in \eqref{eq:cnvgmg:id},
we arrive at
\begin{align}\label{eq:cnvgmg:s}
	&
	\Ex\BK{ \mg_\psi(s)^2 \middle| \scrF(s) }
	=
	\e^3 \alpha\const\eta_\e((1+\alpha\const)^2\xcent)^{-1}  \angleBK{Z^2(s),\psi^2}_\e
	+
	\e^{3} \er(s).
\end{align}
Calculating $ \eta_\e $ to the first order
we have $ \eta_\e = \frac{1+\alpha+\nu\dens+\alpha\dens}{1+\alpha-\nu\dens-\alpha\dens} + O(\e) $.
Using this and \eqref{eq:tcentApprox},
a tedious but straightforward calculation shows that
$ \alpha\const\eta_\e((1+\alpha\const)^2\xcent)^{-1} = (J\tcente)^{-1} + O(\e) $.
Consequently, summing \eqref{eq:cnvgmg:s} over $ s<\te(\tau) $ yields
\begin{align*}
	&
	\sum_{s<\te(\tau)} \Ex\BK{ \mg_\psi(s)^2 \middle| \scrF(s) }
	=
	\frac{\e^{3}}{\tcente J} \sum_{s< \te(\tau)} \angleBK{Z^2(s),\psi^2}_\e
	+
	\er(s).
\end{align*}
The r.h.s.\ represents a discrete approximation
of $ \int_0^\tau \anglebk{ \Ze(\sigma)^2,\psi^2 } d \sigma $,
so following the same procedure as in the proof of \eqref{eq:cnvg:Zt},
one concludes \eqref{eq:cnvg:mg}.
\end{proof}


\bibliographystyle{abbrv}
\bibliography{CT-vertexKPZ}

\end{document}